\documentclass[11pt]{article}
\usepackage[utf8]{inputenc}
\usepackage{amsthm}
\usepackage{amsmath, graphicx, calc}
\usepackage{bbm}
\usepackage{amsfonts, enumitem}
\usepackage{amssymb}
\usepackage[left=2.5cm,right=2.5cm,top=2cm,bottom=2cm]{geometry}

\usepackage{tabto}
\TabPositions{2 cm, 4 cm, 6 cm, 8 cm}

\usepackage{tikz-cd}

\usepackage{tikz}
\usetikzlibrary{arrows, automata, positioning}

\usepackage{tocloft}
\usepackage{appendix}

\usepackage{hyperref}
\hypersetup{
  colorlinks   = true, 
  urlcolor     = blue, 
  linkcolor    = blue, 
  citecolor   = red 
}
 
\newtheorem{theorem}{Theorem}[section]
\newtheorem{corollary}[theorem]{Corollary}
\newtheorem{lemma}[theorem]{Lemma}

\newtheorem{proposition}[theorem]{Proposition}

\theoremstyle{definition}
\newtheorem{definition}[theorem]{Definition}
\newtheorem{example}[theorem]{Example}
\newtheorem*{definition*}{Definition}
\newtheorem{remark}[theorem]{Remark}

\newtheorem*{lemma*}{Lemma}
\newtheorem*{proposition*}{Proposition}
\newtheorem*{theorem*}{Theorem}
\newtheorem*{corollary*}{Corollary}
\newtheorem*{claim*}{Claim}

\theoremstyle{definition}

\newcommand{\Ga}{\Gamma}
\newcommand{\HG}{\mathcal{H}(G)}
\newcommand{\Hl}{\mathcal{H}}
\newcommand{\Z}{\mathbb{Z}}
\newcommand{\N}{\mathbb{N}}
\newcommand{\R}{\mathbb{R}}
\newcommand{\op}{\operatorname} 
\newcommand{\acts}{\curvearrowright}
\newcommand{\nl}{\mathrm{(NL)}}
\newcommand{\ngt}{\mathrm{(NGT)}}
\newcommand{\nne}{\mathrm{(NNE)}}

\title{\textsc{Property $\nl$ for group actions \\ on hyperbolic spaces}}
\author{Sahana H. Balasubramanya \and Francesco Fournier-Facio \and Anthony Genevois \and with Appendix by Alessandro Sisto}
\date{\today}

\begin{document}

\maketitle

\begin{abstract}
We introduce Property $\nl$, which indicates that a group does not admit any (isometric) action on a hyperbolic space with loxodromic elements. In other words, such a group $G$ can only admit elliptic or horocyclic hyperbolic actions, and consequently its poset of hyperbolic structures $\Hl(G)$ is trivial.
It turns out that many groups satisfy this property; and we initiate the formal study of this phenomenon. Of particular importance is the proof of a dynamical criterion in this paper that ensures that groups with ``rich" actions on compact Hausdorff spaces have Property $\nl$. These include many Thompson-like groups, such as $V, T$ and even twisted Brin--Thompson groups, which implies that every finitely generated group quasi-isometrically embeds into a finitely generated simple group with Property $\nl$.
We also study the stability of the property under group operations and explore connections to other fixed point properties.
In the appendix (by Alessandro Sisto) we describe a construction of cobounded actions on hyperbolic spaces starting from non-cobounded ones that preserves various properties of the initial action.
\end{abstract}

\vspace*{-40mm}
\tableofcontents
\addtocontents{toc}{\vspace{-2cm}}

\section{Introduction}

Among the techniques provided by geometric group theory in the study of groups, constructing and exploiting (isometric) actions on hyperbolic spaces is one of the most fruitful and has received a lot of attention in the last decades (see \cite{bestvinasurvey} for a recent survey).
In this article, we focus on groups that are not reachable by such a strategy. The approach can be compared with the study of Kazhdan's Property (T), which forbids isometric actions on Hilbert spaces without global fixpoints, and exhibits interesting behaviours of rigidity. Several families of groups are known to have few or no interesting actions on hyperbolic spaces, including higher rank lattices \cite{haettel, lattices}, Thompson's group $V$ \cite{anthony}, and Gromov's random monsters \cite{GromovMonster}, but so far a systematic study of such behaviours has not been conducted. In this spirit, we introduce the following property and dedicate most of this article to its study.

\begin{definition}
A group $G$ satisfies \emph{Property $\nl$} -- standing for \emph{No Loxodromics} -- if no action by isometries of $G$ on a hyperbolic space admits a loxodromic element. A group has Property \emph{Hereditary $\nl$} if all of its finite-index subgroups have Property $\nl$. We will often simply say that $G$ is $\nl$, or hereditary $\nl$.
\end{definition}

In other words, a group is $\nl$ if its only possible actions on hyperbolic spaces are elliptic or horocyclic (see Theorem \ref{thm:ClassHypAct} below). These are the only types of actions we cannot forbid since every group admits elliptic actions (e.g.\ the trivial action on a single point) and every countably infinite group admits horocyclic actions (e.g.\ combinatorial horoballs on Cayley graphs \cite{GrovesManning} or trees in case the group is not finitely generated \cite{Serre}). Thus, a group is $\nl$ if it admits as few types of actions on hyperbolic spaces as possible. 

While Property $\nl$ is preserved by certain group operations like directed unions, direct sums, extensions, etc (see Theorem \ref{intro:thm:stability}) it is sensitive to commensurability, i.e.\ a group may be $\nl$ and contain finite-index subgroups with rich actions on hyperbolic spaces (see Example \ref{ex:nlnothere}). This is the first motivation in the introduction of property hereditary $\nl$ -- it turns out that the hereditary version of the property is also preserved under many group operations, including commensurability.  A second one is that this is the re-enforcement needed to imply fixed-point theorems in spaces ``built from hyperbolic spaces'' (see Section \ref{sec:otherprops}). 

An additional motivator is the relation to the structure of the poset $\HG$ - the poset of hyperbolic structures first introdoced in \cite{ABO} (see Section \ref{sec:otherprops}). It is easy to see that a group with Property $\nl$ has trivial $\HG$ poset.  In general, the structure of this poset is not stable under taking finite-index subgroups (see Example \ref{ex:finindexposet}), but this issue does not occur with the property of being hereditary $\nl$: such a group and all its finite-index subgroups have trivial $\HG$.

\subsection{First examples}

We start by mentioning a few examples that are already available in the literature. First of all, since loxodromic elements must have infinite order, every torsion group is $\nl$, and even hereditary $\nl$. This includes finite groups, Burnside groups, several branch groups such as Grigorchuk's group, and Tarski monsters.

A more interesting class of examples can be found in higher-rank lattices:

\begin{theorem}[\cite{haettel}]
Let $G$ be a product of higher rank algebraic groups over local fields, each of which is almost simple and has finite center. Then every lattice in $G$ is hereditary $\nl$.
\end{theorem}

In \cite{haettel} the theorem is stated in a manner equivalent to Property $\nl$, but since a finite-index subgroup of a lattice is a lattice, the hereditary property follows directly. This result has been strengthened in \cite{lattices}, where the authors show that, roughly, in the presence of rank-$1$ factors, the only possible actions on a hyperbolic spaces are those that factor through them.

In another direction is the following result:

\begin{theorem}[\cite{anthony}]
Thompson's group $V$ is hereditary $\nl$.
\end{theorem}

In \cite{anthony}, the result is stated in a manner that excludes general type actions for $V$ (see Theorem \ref{thm:ClassHypAct} for classification of hyperbolic actions). However Property $\nl$ easily follows from the fact that $V$ is uniformly perfect (see Proposition \ref{prop:ngt:to:nl}), and hereditary $\nl$ then follows from the fact that $V$ is simple. The strategy used to prove this result is part of a much wider picture, that we will expand and present in this paper (see Section \ref{thompson}).

For amenable groups, the absence of free subgroups easily implies that they cannot admit general type actions. From here, we exploit the rigidity properties of their quasimorphisms in order to obtain certain conditions that describe when amenable groups admit hyperbolic actions with loxodromics. For focal or oriented lineal actions, this can be classicaly achieved via the \emph{Busemann quasimorphism} \cite{Man}. In this paper we generalize this to encompass non-oriented lineal actions: the resulting object is called a \emph{Busemann quasicocycle} (see Proposition \ref{prop:quasicocycle}). Some of the results we deduce from this are as follows.

\begin{theorem}[Proposition \ref{prop:Amenable}]
\label{intro:thm:amenable}

Let $G$ be an amenable group. Then $G$ is $\nl$ if and only if every homomorphism $G \to \mathbb{R} \rtimes \mathbb{Z}/2\mathbb{Z}$ has image of order at most $2$. In particular:
\begin{enumerate}[noitemsep]
    \item An abelian group is $\nl$ if and only if it is torsion.
    \item A nilpotent group is $\nl$ if and only if every homomorphism to $\mathbb{R}$ is trivial.
\end{enumerate}
\end{theorem}

\begin{corollary}
Let $G$ be a finitely generated amenable group. Then $G$ is $\nl$ if and only if it does not surject onto $\mathbb{Z}$ or $\mathbb{D}_\infty$.
\end{corollary}

\subsection{Groups of homeomorphisms}

The main class of examples that is explored in this paper comes from groups of homeomorphisms. Given a compact Hausdorff space $X$ and a group $G$ acting faithfully by homeomorphisms on $X$, we produce a criterion that ensures that $G$ cannot admit general type (isometric) actions on hyperbolic spaces under the assumption that the action has certain dynamical properties (see Theorem \ref{thm:cH:criterion}). This is a dynamical version of an algebraic criterion from \cite{anthony}, which in certain cases can be combined with uniform perfection to obtain strong results about certain groups having $\nl$. The statements are natural, but a little too technical for this introduction, so we refer the reader to Section \ref{thompson} for the general results. Here, we limit ourselves to a list of examples we produce in this paper.

\begin{theorem}[Section \ref{thompson}]
\label{intro:thm:thompson}

The following groups are hereditary $\nl$:
\begin{enumerate}[noitemsep]
    \item Higman--Thompson groups $V_n(r)$;
    \item Some R\"{o}ver--Nekrashevych groups, including Neretin's group;
    \item Twisted Brin--Thompson groups $SV_\Gamma$ whenever $S$ is a countable faithful $\Gamma$-set;
    \item Higman--Stein--Thompson groups $T_{n_1, \ldots, n_k}$, with $k \geq 1$ and $n_1 = 2, n_2 = 3$;
    \item The golden ratio Thompson group $T_\tau$;
    \item The finitely presented group $S$ of piecewise projective homeomorphisms of the circle, from \cite{yashS};
    \item Symmetrizations of Higman--Thompson groups $QV_n(r)$ and $QT$.
\end{enumerate}
\end{theorem}

In particular, the result on twisted Brin--Thompson groups shows that there are plenty of hereditary $\nl$ groups, in the following sense.

\begin{corollary}\label{cor:QI-embed}
Every finitely generated group quasi-isometrically embeds into a finitely generated simple group, which is moreover hereditary $\nl$.
\end{corollary}

We warn the reader that our criterion ensures Property $\nl$ under the existence of flexible enough actions on compact Hausdroff spaces, and it is not saying that in general all ``Thompson-like'' groups have Property $\nl$. For instance, Braided Thompson groups have rich actions on hyperbolic spaces \cite{bV}.

\subsection{Stability properties}

An advantage of our systematic approach is that it naturally leads to the study of the stability of these properties under natural group-theoretic operations. The following theorem summarizes some of our results from Section \ref{s:stability}.

\begin{theorem}[Section \ref{s:stability}]
\label{intro:thm:stability}

    The class of (hereditary) $\nl$ groups is closed under the following operations:
    \begin{enumerate}[noitemsep]
        \item Quotients;
        \item Directed unions;
        \item Extensions;
        \item Certain permutational wreath products.
    \end{enumerate}
The property hereditary $\nl$ is also preserved under commensurability.
\end{theorem}

The case of permutational wreath products is analyzed in detail, and the precise conditions that ensure that the resulting group is (hereditary) $\nl$ are both necessary and sufficient. They are expressed in terms of existence of quasimorphisms of the base group with some specific properties (see Section \ref{s:wreath}).

These stability results are useful to construct some further examples: indeed, the stability under extensions is explicitly used to cover some of the items of Theorem \ref{intro:thm:thompson}.

\subsection{Connection to other properties}

Throughout this paper an important role will be played by two weaker properties: $\ngt$ -- for \emph{No General Type} -- which prevents the existence of general type actions, and $\nne$ -- for \emph{No Non Elementary} -- which prevents the existence of non-elementary actions. In certain situations, such as the dynamical criterion, establishing these properties first is more natural, and the passage from Property $\ngt$ to $\nne$ to $\nl$ can be done separately by means of more analytic considerations, mostly relating to the absence of unbounded quasimorphisms.

The last section of this paper aims at initiating the discussion of Property (hereditary) $\nl$ as it relates to other popular areas of study in geometric group theory. This includes fixed point properties like Property (FA), Property ($\mathrm{F}\mathbb{R}$), hierarchical hyperbolicity, and actions on products of hyperbolic spaces. While we study some implications -- mostly to record our observations and serve as motivation -- a full investigation of the extent of this connection is beyond the scope of this paper; and so we leave it to future work to explore more thoroughly. 

Lastly, we study the relationship between these properties and the poset of hyperbolic structures introduced in \cite{ABO}, in particular we will see that Property $\nl$ is equivalent to a version of Property $\nl$ for cobounded actions. This will follow from the work contained in the appendix -- a general result about associating cobounded actions to arbitrary actions on hyperbolic spaces, while preserving several useful properties, which is also of independent interest.

\paragraph{Acknowledgements.} The authors wish to thank Yash Lodha for asking the question that motivated this paper. They are indebted to him, Waltraud Lederle, Nicol{\'a}s Matte Bon, Romain Tessera, Abdul Zalloum and Matt Zaremsky for useful conversations. We also thank the anonymous referee for a thorough reading of the paper and for making several useful suggestions. The first author was supported by the Deutsche Forschungsgemeinschaft (DFG, German Research Foundation) -Project-ID 427320536 – SFB 1442, as well as under Germany's Excellence Strategy EXC 2044 390685587, Mathematics Münster: Dynamics–Geometry–Structure. 

\paragraph{Outline of the paper.} We start with some preliminaries in Section \ref{sec:prelims}, covering the definitions and relations between the properties $\ngt, \nne$ and $\nl$. In Section \ref{s:amenable} we investigate amenable groups, and prove Theorem \ref{intro:thm:amenable}. In Section \ref{s:stability} we study stability properties and prove Theorem \ref{intro:thm:stability}. The dynamical criterion and the examples from Theorem \ref{intro:thm:thompson} are dealt with in Section \ref{thompson}. Finally, we explore connections to other properties in Section \ref{sec:otherprops}. The appendix outlines a method to pass from arbitrary actions to cobounded ones.

\section{Preliminaries}\label{sec:prelims}

Unless mentioned otherwise, all actions on hyperbolic spaces considered in this paper are by isometries. However, in Section \ref{thompson} we will consider groups acting by homeomorphisms on compact Hausdorff spaces. In this section, we shall recall some necessary background information as well as state some basic results and examples about the properties considered in this paper. These results will also be of relevance in later sections of this paper. 

We begin by recalling some standard facts and definitions pertaining to isometric group actions on hyperbolic spaces, which we sometimes call \emph{hyperbolic actions}. Given a hyperbolic space $X$, we denote by $\partial X$ its Gromov boundary. In general, $X$ is not assumed to be proper, and its boundary is defined as the set of equivalence classes of sequences convergent at infinity. Given a group $G$ acting on a hyperbolic space $X$,  we denote by $\Lambda (G)$ the set of limit points of $G$ on $\partial X$. That is, $$\Lambda (G)=\partial X\cap \overline{Gx},$$ where $\overline{Gx}$ denotes the closure of a $G$--orbit in $X\cup \partial X$, for any choice of basepoint $x\in X$.  This definition is independent of the choice of $x\in X$; for details the reader is referred to \cite{Gro}.  The action of $G$ is called \emph{elementary} if $|\Lambda (G)|\le 2$ and \emph{non-elementary} otherwise. The action of $G$ on $X$ naturally extends to a continuous action of $G$ on $\partial X$.

\subsection{Classification of isometries and hyperbolic actions} Given any action of a group $G$ on a hyperbolic space $X$, isometries can be of exactly one of the following three types. An element $g\in G$ is called
\begin{itemize}[noitemsep]
\item[(i)] \emph{elliptic} if $\langle g\rangle $ has bounded orbits;
\item[(ii)] \emph{loxodromic} if the map $n \mapsto g^nx, n \in \Z$ is a quasi-isometric embedding for some (equivalently any) $x \in X$; 
\item[(iii)] \emph{parabolic} otherwise.
\end{itemize} 

Every loxodromic element $g\in G$ has exactly $2$ fixed points $g{^{\pm \infty}}$ on $\partial X$, where $g{^{+\infty}}$ (respectively, $g{^{-\infty}}$) is the limit of the sequence $(g{^{n}}x)_{n\in \mathbb N}$ (respectively, $(g{^{-n}}x)_{n\in \mathbb N}$) for any fixed $x\in X$. Thus $\Lambda (\langle g\rangle) =\{ g{^{\pm \infty}}\}$. In a similar vein, a parabolic element has exactly one fixed point $\xi$ on $\partial X$, which is the limit of both sequences $(g{^{n}}x)_{n \in \mathbb{N}}$ and $(g{^{-n}}x)_{n \in \mathbb{N}}$. 

The following theorem summarizes the standard classification of groups acting on hyperbolic spaces due to Gromov \cite[Section 8.2]{Gro} and the results  \cite[Propositions 3.1 and 3.2]{Amen}.

\begin{theorem}\label{thm:ClassHypAct}
Let $G$ be a group acting on a hyperbolic space $X$. Then exactly one of the following conditions holds.
\begin{itemize}
\item[1)] $|\Lambda (G)|=0$. Equivalently,  $G$ has bounded orbits. In this case the action of $G$ is called \emph{elliptic}.

\item[2)] $|\Lambda (G)|=1$. In this case the action of $G$ is called \emph{horocyclic} (or \emph{parabolic}). A horocyclic action cannot be cobounded. 

\item[3)] $|\Lambda (G)|=2$. Equivalently, $G$ contains a loxodromic element and any two loxodromic elements have the same limit points on $\partial X$. In this case the action of $G$ is called \emph{lineal}. A lineal action of a group $G$ on a hyperbolic space is \emph{orientable} if no element of $G$ permutes the two limit points of $G$ on $\partial X$, and \emph{non-orientable} otherwise.

\item[4)] $|\Lambda (G)|=\infty$. Then $G$ always contains loxodromic elements. In this case the action of $G$ is called \emph{non-elementary}. In turn, this case breaks into two subcases.
\begin{itemize}
\item[(a)] $G$ fixes a point of $\partial X$. Equivalently, any two loxodromic elements of $G$ have a common limit point on the boundary. In this case the action of $G$ is called \emph{focal} (or \emph{quasi-parabolic}). 
\item[(b)] $G$ does not fix any point of $\partial X$. In this case the action of $G$ is said to be of \emph{general type}.
\end{itemize}
\end{itemize}
\end{theorem}

\begin{remark} In the case of non-orientable lineal actions, the action on the set of limit points of $G$ induces a surjective homomorphism $G \to \Z/ 2\Z$. The kernel of this homomorphism has index $2$, and the restriction of the action to this kernel is lineal and orientable.
\end{remark}

In addition to Property $\nl$ defined in the introduction, we are also interested in the following related properties.

\begin{definition}
\label{def:ngt:nne}

We say that a group has property 
\begin{itemize}[noitemsep]
	\item[1.] $\ngt$ if it does not admit an action of general type on a hyperbolic space;
	\item[2.] $\nne$ if it does not admit a non-elementary action on a hyperbolic space.
\end{itemize}
\end{definition} 

In other words, $\ngt$ excludes general type actions and $\nne$ excludes focal and general type actions. These can be thought of as fixed point properties as either the group has a bounded orbit (in case of elliptic actions) or it has an orbit in the boundary of size at most $2$ (in case of lineal or focal actions). Property $\ngt$ was considered in \cite{anthony} with the name \emph{hyperbolically elementary}.

\subsection{The Busemann quasimorphism} 

A function $q\colon G\to \mathbb R$ is a \emph{quasimorphism} (or \emph{quasicharacter}) if there exists a constant $D$ such that $|q(gh)-q(g)-q(h)|\le D$ for all $g,h\in G$. The infimum of such $D$ is called the \emph{defect} of $q$ and is denoted $D(q)$; in particular $D(q) = 0$ if and only if $q$ is a homomorphism. If, in addition, the restriction of $q$ to every cyclic subgroup of $G$ is a homomorphism, $q$ is called a \emph{homogeneous quasimorphism} (or \emph{pseudocharacter}). Every quasimorphism $q$ gives rise to a homogeneous quasimorphism $\beta$ defined by the following limit, which always exists. For any $g\in G$:
$$
\beta(g):=\lim_{n\to \infty} \frac{q(g{^{n}})}n.
$$
The function $\beta$ is called the \emph{homogenization of $q$.} It is straightforward to check that
$|\beta(g) -q(g)|\le D(q)$ for all $g\in G$. In particular $\beta$ is also a quasimorphism, and moreover it follows directly from the definition that it is homogeneous \cite[Section 2.2.2]{scl}. Note moreover that any bounded homogeneous quasimorphism is trivial, therefore it follows that the homogenization of $q$ is the \emph{unique} homogeneous quasimorphism at a bounded distance from $q$.

In some sense, unbounded quasimorphisms on a group detect properties of negative curvature of the group, as discussed below. In the opposite direction, we have the following:

\begin{proposition}[see e.g. {\cite[Proposition 2.65]{scl}}]
\label{prop:amenable:qm}

Let $G$ be an amenable group. Then every homogeneous quasimorphism on $G$ is a homomorphism.
\end{proposition}

We will discuss hyperbolic actions of amenable groups in detail in Section \ref{s:amenable}, the reader is referred there for reminders about the definition of amenability and relevant subclasses.

Given any action of a group $G$ on a hyperbolic space $X$ fixing a point $\xi$ on the boundary, one can associate a natural homogeneous quasimorphism $\beta$ called the \emph{Busemann quasimorphism (based at $\xi$)}. We will briefly go over the definition and main properties of $\beta$, and refer the reader to \cite[Sec. 7.5.D]{Gro} and \cite[Sec. 4.1]{Man} for further details.

Let $\mathbf{x} = (x_n)_{n \geq 0}$ be a sequence converging to $\xi \in \partial X$. We define
$q_{\mathbf{x}} : G \to \mathbb{R}$ by the rule
$$q_{\mathbf{x}}(g) = \limsup\limits_{n \to \infty} (d(g x_0, x_n) - d(x_0, x_n)).$$

\begin{proposition}[{\cite[Sec. 4.1]{Man}}]
\label{prop:busemann}

With the above setup, the following properties hold:
\begin{enumerate}
    \item[(1)] $q_{\mathbf{x}}$ is a quasimorphisms.
    \item[(2)] The homogenization of $q_{\mathbf{x}}$ is independent of $\mathbf{x}$: we call it the \emph{Busemann quasimorphism} and denote it by $\beta_\xi$, or simply by $\beta$ when it is clear from the context.
    \item[(3)] $\beta(g) \neq 0$ if and only if $g$ is loxodromic. In particular, $\beta$ is not identically zero whenever $G \curvearrowright X$ is focal or orientable lineal.
\end{enumerate}
\end{proposition}

If $\beta$ is a homomorphism, then the action $G\curvearrowright X$ is called \emph{regular}.

\medskip
Conversely, given a quasimorphism on a group $G$, one can always construct an orientable lineal action. 

\begin{lemma}[{\cite[Lemma 4.15]{ABO}}]\label{lem:constrlin} Let $p\colon G \to \R$ be an unbounded homogeneous quasimorphism. Let $C$ be any constant such that $D(p) \leq C/2$ and there exists a value of $p$ in the interval $(0,C/2)$. Let $$ X =X_{p,C} =\{g \in G: |p(g)| < C\}.$$ Then $X$ generates $G$ and the map $p \colon (G,d_X) \to \R$ is a quasi-isometry. In particular, if $\Ga(G,X)$ denotes the Cayley graph of $G$ with respect to $X$, $G \acts \Gamma(G,X)$ is an orientable lineal action.
\end{lemma}

An immediate corollary is the following implication for groups with Property $\nl$. 

\begin{corollary}
\label{cor:nl:noqm}

If $G$ has Property $\nl$, then $G$ has no unbounded quasimorphisms. In particular, if $G$ is finitely generated and $\nl$, then $G$ has finite abelianization. 
\end{corollary}

\subsection{The Busemann quasicocycle}

The Busemann quasimorphism is classically defined only for actions fixing a point in the boundary. Here we introduce a generalization to non-orientable lineal actions, where instead two points on the boundary are fixed setwise. This entails the notion of a \emph{quasicocycle}.

\begin{definition}
Let $\varepsilon : G \to (\{ \pm 1 \}, \times)$ be a homomorphism. An \emph{$\varepsilon$-quasicocycle} is a map $\varphi : G \to \mathbb{R}$ such that there exists a constant $D$ with
$$|\varphi(gh) - \varphi(g) - \varepsilon(g)\varphi(h)| \leq D$$
for all $g, h \in G$. The infimum of such $D$ is called the \emph{defect of $\varphi$} and is denoted $D(\varphi)$.
\end{definition}

Note that if $\varepsilon$ is the trivial homomorphism, we recover the usual definition of quasimorphism. Just as with quasimorphisms and homomorphisms, there is a simplest type of $\varepsilon$-quasicocycle:

\begin{lemma}
\label{lem:cocycle}

Let $\varphi$ be an $\varepsilon$-quasicocycle. Then $D(\varphi) = 0$ if and only if the map $G \to \mathbb{R} \rtimes \{ \pm 1 \} : g \mapsto (\varphi(g), \varepsilon(g))$ is a homomorphism. In this case, $\varphi$ is called an \emph{$\varepsilon$-cocycle}.
\end{lemma}

The proof of the above lemma is an immediate application of the definition, and we leave it to the reader as an exercise. This setup allows us to get the following proposition. 

\begin{proposition}
\label{prop:quasicocycle}

Let $G$ be a group with a non-orientable lineal action on the hyperbolic space $X$, with limit points $\xi_{\pm}$. Define $\varepsilon(g) = -1$ if $g$ swaps $\xi_{\pm}$, and $\varepsilon(g) = 1$ otherwise. Let $K$ be the kernel of $\varepsilon$. Then the Busemann quasimorphism $\beta_{\xi_+} : K \to \mathbb{R}$ extends to an $\varepsilon$-quasicocycle $\varphi$ on $G$ of defect at most $2 D(\beta_{\xi_+})$. In particular, $\varphi$ is a cocycle if and only if the action of $K$ on $X$ is regular.
\end{proposition}

\begin{proof}
Since we are assuming the action to be non-orientable, we can fix an element $s \in G$ such that $\varepsilon(s) = -1$. To start, we will show that for every $g \in K$, it holds $\beta_{\xi_+}(sgs^{-1}) = -\beta_{\xi_+}(g)$ (notice that $K$ is normal in $G$, thus $sgs^{-1} \in K)$. If $g$ is not loxodromic, then neither is $sgs^{-1}$, so the above equality follows from Proposition \ref{prop:busemann}(3). Otherwise, if $g$ is loxodromic, for a fixed $x \in X$ we have $g^{\pm n} x \to \xi_{\pm}$ (up to replacing the indices on $\xi_{\pm}$). Denote these sequences by $\mathbf{x}_{\pm}$. We then compute, for $k \geq 1$:
\begin{align*}
    q_{\mathbf{x}_-}(g^k) &= \limsup\limits_{n \to \infty} d(g^k x, g^{-n} x) - d(x, g^{-n} x) \\
    &= \limsup\limits_{n \to \infty} d(g^n x, g^{-k} x) - d(g^n x, x) \\
    &= q_{\mathbf{x}_+}(g^{-k}).
\end{align*}
On the other hand
\begin{align*}
    q_{\mathbf{x}_+}(sg^ks^{-1}) &= \limsup\limits_{n \to \infty} d(s g^k s^{-1}x, g^n x) - d(x, g^n x) \\
    &= \limsup\limits_{n \to \infty} d(g^ks^{-1} x, s^{-1} g^n x) - d(s^{-1} x, s^{-1} g^n x) \\
    &= q_{s^{-1} \mathbf{x}_+}(g^k).
\end{align*}
By Proposition \ref{prop:busemann}(2), we use can both $s^{-1} \mathbf{x}_+$ and $\mathbf{x}_-$ to compute $\beta_{\xi_-}$ (since both sequences converge to $\xi_-$). Using the above computations, we obtain
\begin{align*}
    \beta_{\xi_+}(sgs^{-1}) &= \lim\limits_{k \to \infty} \frac{q_{\mathbf{x}_+}(sg^ks^{-1})}{k} = \lim\limits_{k \to \infty} \frac{q_{s^{-1} \mathbf{x}_+}(g^k)}{k} \\
    &= \beta_{\xi_-}(g) = \lim\limits_{k \to \infty} \frac{q_{\mathbf{x}_-}(g^k)}{k} = \lim\limits_{k \to \infty} \frac{q_{\mathbf{x}_+}(g^{-k})}{k} \\
    &= \beta_{\xi_+}(g^{-1}) = -\beta_{\xi_+}(g),
\end{align*}
as desired. As a corollary, we obtain
$$\beta_{\xi_+}(s^2) = \beta_{\xi_+}(s s^2 s^{-1}) = - \beta_{\xi_+}(s^2)$$
and thus $\beta_{\xi_+}(s^2) = 0$ (notice that $K$ has index $2$ in $G$, so $s^2 \in K$).

Now we are ready to define the quasicocycle. We set $\beta := \beta_{\xi_+}$ for simplicity for the rest of the proof. Each element in $G$ may be uniquely written as $k s^i$ where $k \in K$ and $i \in \{ 0, 1 \}$. We set $\varphi(ks^i) := \beta(k)$, and we claim that this is an $\varepsilon$-quasicocycle, where $\varepsilon$ is as defined in the statement of the proposition. Let $g_1 = k_1 s^{i_1}, g_2 = k_2 s^{i_2} \in G$. We compute
$$\varphi(g_1 g_2) = \varphi(k_1 s^{i_1} . k_2 s^{i_2}) = \varphi(k_1 . s^{i_1} k_2 s^{-i_1} . s^{i_1 + i_2}).$$
Suppose first that $i_1 +i_2 = 0 \text{ or } 1$. Then it is easy to check that the above equals $$\beta(k_1 . s^{i_1} k_2 s^{-i_1}) \sim \beta(k_1) + \varepsilon(s^{i_1}) \beta(k_2) = \varphi(g_1) + \varepsilon(g_1) \varphi(g_2),$$
where $\sim$ denotes an error of at most $D(\beta)$. Otherwise, if $i_1 = i_2 = 1$, then the above equals
$$\beta(k_1 . s k_2 s^{-1} . s^2) \sim \beta(k_1) + \varepsilon(s) \beta(k_2) + \beta(s^2) = \varphi(g_1) + \varepsilon(g_1) \varphi(g_2) + 0,$$
where $\sim$ denotes an error of at most $2 D(\beta)$. We conclude that $\varphi$ is a quasicocycle of defect at most $2 D(\beta)$. Lastly, $\beta$ is the restriction of $\varphi$ to $K$, and so $D(\beta) \leq D(\varphi)$. In particular $D(\varphi) = 0$ if and only if $D(\beta) = 0$, so the last statement follows from Lemma \ref{lem:cocycle}.
\end{proof}

We record the following special case.

\begin{corollary}\label{cor:LinealAmenable}
Let $G$ be a group with the property that every homogeneous quasimorphism on $G$ or an index-$2$ subgroup of $G$ is a homomorphism. Then every lineal action of $G$ on a hyperbolic space $X$ defines a homomorphism $\psi : G \to \mathbb{R} \rtimes \mathbb{Z}/2\mathbb{Z}$ such that
\begin{enumerate}
    \item If the image of $\psi$ has order greater than $2$, then it contains a non-trivial (thus unbounded) subgroup of $\mathbb{R}$;
    \item The projection of $\psi(g)$ onto $\mathbb{Z}/2\mathbb{Z}$ is non-trivial if and only if $g$ swaps the pair of points in $\partial X$ fixed by $G$.
\end{enumerate}
In particular, this holds if $G$ is amenable.
\end{corollary}

\begin{proof}
This is a direct consequence of Proposition \ref{prop:quasicocycle}, together with the fact that every subgroup of $\mathbb{R} \rtimes \mathbb{Z}/2\mathbb{Z}$ of order greated than $2$ intersects $\mathbb{R}$ non-trivially. The last statement then follows from Proposition \ref{prop:amenable:qm}, and the fact that (index-$2$) subgroups of amenable groups are amenable.
\end{proof}

\begin{remark}
While the Busemann quasimorphism is canonically defined (Proposition \ref{prop:busemann}), the Busemann quasicocycle we constructed in Proposition \ref{prop:quasicocycle} is not. Indeed, we could have chosen any element $s' \in G$ swapping $\xi_{\pm}$, and this would have lead to an $\varepsilon$-quasicocycle $\varphi'$ such that $\varphi'(s') = 0$.
This choice is due to the fact that every identification of $\mathrm{Isom}(\mathbb{R})$ with $\mathbb{R} \rtimes \mathbb{Z} / 2 \mathbb{Z}$ comes with a choice of basepoint, i.e. of the fixpoint of the distinguished $\mathbb{Z} / 2 \mathbb{Z}$ factor.
\end{remark}

\begin{remark}
There is a strong relation between quasimorphisms and \emph{bounded cohomology} \cite{frigerio, scl}, namely quasimorphisms modulo homomorphisms and bounded functions -- or equivalently homogeneous quasimorphisms modulo homomorphisms -- represent those classes in $H^2_b(G; \mathbb{R})$ that lie in the kernel of the natural \emph{comparison map} $H^2_b(G; \mathbb{R}) \to H^2(G; \mathbb{R})$. Similarly, given a homomorphism $\varepsilon : G \to \mathbb{Z}/2\mathbb{Z}$, we denote by $\mathbb{R}_{\varepsilon}$ the $G$-module $\mathbb{R}$ endowed with the linear isometric action defined by $\varepsilon$. Then $\varepsilon$-quasicocycles modulo $\varepsilon$-cocycles and bounded functions represent those classes in $H^2_b(G; \mathbb{R}_{\varepsilon})$ that lie in the kernel of the natural comparison map $H^2_b(G; \mathbb{R}_{\varepsilon}) \to H^2(G; \mathbb{R}_{\varepsilon})$.

It follows that the conclusion of Corollary \ref{cor:LinealAmenable} also holds under the assumption that this comparison map is injective, without assuming anything about finite-index subgroups. This is the case, for instance, for all amenable groups \cite{frigerio}, some high rank lattices \cite{bm1, bm2}, and some lamplighters and Thompson groups \cite{lamplighters}. However, we preferred to state Corollary \ref{cor:LinealAmenable} in a way that requires no understanding of bounded cohomology, and that is still sufficient for our purposes.
\end{remark}

Further in the non-orientable case, the correspondence between lineal actions and quasicocycles continues to hold. The following lemma is an analogue of Lemma \ref{lem:constrlin}.

\begin{lemma}
\label{lem:constrlin2}

Let $\varepsilon : G \to \{ \pm 1 \}$ be a homomorphism with kernel $K$, and let $\varphi : G \to \mathbb{R}$ be an unbounded $\varepsilon$-quasicocycle, which restricts to a homogeneous quasimorphism on $K$. Let $C > 0$ be such that $D(\varphi) \leq C/2$ and $\varphi(K) \cap (0, C/2) \neq \varnothing$ and $\varphi(G \setminus K) \cap [0, C/2) \neq \varnothing$. Let $X := \{ g \in G : |\varphi(g)| < C \}$. Then $X$ generates $G$ and the map $\varphi : (G, d_X) \to \mathbb{R}$ is a quasi-isometry.
\end{lemma}

Note that any $\varepsilon$-quasicocycle is at a bounded distance from one that restricts to a homogeneous quasimorphism on the kernel of $\varepsilon$: it suffices to replace $\varphi$ with its homogenization on $K$, and keep it equal on $G \setminus K$.

\begin{proof}
Let $Y := X \cap K$. By Lemma \ref{lem:constrlin} \cite[Lemma 4.15]{ABO}, we know that $Y$ generates $K$, and that the map $\varphi : (K, d_Y) \to \mathbb{R}$ is a quasi-isometry. By assumption, there exists $s \in G \setminus K \cap X$, and every element of $G$ may be uniquely written as $k.s^i$ for $k \in K, i \in \{0, 1\}$. It follows that $X$ generates $G$, and
$$|k.s^i|_X \leq |k|_X + 1 \leq |k|_Y + 1 << |\varphi(k)| \leq |\varphi(k.s^i)| + D(\varphi) + |\varphi(s)|,$$ where $<<$ denotes the inequality from the quasi-isometry $\varphi$ to $\R$.
For the other direction, let $g \in G$ and let $g = x_1 \cdots x_n$ be an expression in $X$ of shortest length. Then
\begin{align*}
|\varphi(g)| &= |\varphi(x_1 \cdots x_n)| \leq |\varphi(x_1 \cdots x_{n-1})| + |\varphi(x_n)| + D(\varphi) \leq \cdots \\
& \leq n \max_i |\varphi(x_i)| + (n-1) D(\varphi) \leq 2C|g|_X.
\end{align*}
\end{proof}

Similarly to Corollary \ref{cor:nl:noqz} we obtain:

\begin{corollary}
\label{cor:nl:noqz}

If $G$ has Property $\nl$, then $G$ has no unbounded quasicocycle.
\end{corollary}

\subsection{Relations between the properties} As mentioned before, we are also interested in the properties $\ngt$ and $\nne$ (see Definition \ref{def:ngt:nne}). Clearly hereditary $\nl$ implies $\nl$ and it follows easily from Theorem \ref{thm:ClassHypAct}, that $\nl \implies \nne \implies \ngt$. We record additional relations between these properties in this section, which will also be used later in this paper. 

\begin{proposition}[From $\ngt$ to $\nl$]
\label{prop:ngt:to:nl}
Let $G$ be a group with Property $\ngt$. If $G$ admits no unbounded quasimorphisms, then $G$ has Property $\nne$. If moreover no index-$2$ subgroup of $G$ admits unbounded quasimorphisms, then $G$ has Property $\nl$.
\end{proposition}

\begin{proof}
Let $X$ be a hyperbolic space on which $G$ acts. Since $G$ has property $\ngt$, the action is elliptic, horocyclic, lineal, or focal. In the focal case, $G$ fixes a point on $\partial X$ and has a loxodromic element, therefore the Busemann quasimorphism is an unbounded quasimorphism on $G$; a contradiction. In the lineal case, $G$ preserves a pair of points on $\partial X$, so it admits a subgroup $K \leq G$ of index at most $2$ that fixes both. This latter subgroup now fixes a point at infinity and has a loxodromic element, therefore once again the Busemann quasimorphism is an unbounded quasimorphism on $K$; a contradiction.
\end{proof}

Cases of special interest for us, especially in Section \ref{thompson}, will be groups that fall into the following framework. 

\begin{definition}
Let $n \geq 1$. A group $G$ is said to be \emph{$n$-uniformly perfect} if every element in $G$ can be written as a product of at most $n$ commutators. It is said to be \emph{$n$-uniformly simple} if for every pair of elements $g, h \in G$, $h$ can be written as a product of at most $n$ conjugates of $g, g^{-1}$.

We say that $G$ is uniformly perfect (resp. uniformly simple) if there exists $n \geq 1$ such that $G$ is $n$-uniformly perfect (resp. $n$-uniformly simple).
\end{definition}

It is straightforward to verify that every uniformly perfect group is perfect, and every uniformly simple group is both simple and uniformly perfect (as one can take $g$ to be a commutator in the definition). 

\begin{corollary}\label{cor:nl:uperfect}
Let $G$ be a group with Property $\ngt$. If $G$ is uniformly perfect, then $G$ has $\nl$. If $G$ is uniformly simple, then $G$ is hereditary $\nl$.
\end{corollary}

\begin{proof}
Every quasimorphism of a uniformly perfect group is bounded \cite[Section 2.2.3]{scl}, so the first statement follows from Proposition \ref{prop:ngt:to:nl}. Indeed, $\nne$ follows immediately. It is easy to check that since index 2-subgroups are normal, such a subgroup would also be uniformly perfect and thus have no unbounded quasimorphisms, implying Property $\nl$. Lastly, a uniformly simple group is uniformly perfect, and being simple has no finite-index subgroups, so hereditary $\nl$ follows automatically. \end{proof}

\subsection{Factorisation of actions} A common problem considered in geometric group theory is to reduce a given action of a group to a well-understood action of the same group or a subgroup. A related problem is the \emph{extension problem}, where one tries to build a hyperbolic action of $G$ starting from a hyperbolic action of a subgroup $H$, such that the extended action still ``witnesses'' the action of $H$. In this paper, we shall make use of the following notions of \emph{factorization of actions} and  \emph{essential} actions.  

In what follows, a subspace $X'$ of a geodesic metric space $(X,d)$ is called \emph{quasiconvex} if there is a constant $\lambda \geq 0$ such that for every $x,y \in X'$, every geodesic connecting $x,y$ in $X$ lies in the $\lambda$-neighborhood of $X'$. It is well known that quasiconvex subspaces of hyperbolic spaces are themselves hyperbolic. A subspace $X' \subset X$ is called \emph{quasi-dense} if there is a constant $C \geq 0$ such that for every $x \in X$, there is a $y \in X'$ such that $d(x,y) \leq C$. 

\begin{definition}\label{def:factoringaction}
Let $G,H$ be two groups and $\varphi : G \to H$ a homomorphism. An action $G\curvearrowright X$ \emph{factors through $\varphi$} if there exist a $G$-invariant quasiconvex $X' \subset X$, an action $H \curvearrowright Y$, and a $\varphi$-equivariant quasi-isometry $X' \to Y$. 
\end{definition}

In other words, an action of $G$ factors through $H$ via $\varphi$ up to perturbing the hyperbolic space a little bit and extracting an invariant quasiconvex subspace. 

\begin{definition}
An action on a hyperbolic space is \emph{essential} if every invariant quasiconvex subspace is quasi-dense.
\end{definition}

It is worth noticing that hyperbolic actions can be assumed to be essential without loss of generality. Indeed, the following lemma is well-known to experts; we add a sketch of the proof for the reader's convenience.

\begin{lemma}\label{lem:Essential}
Let $G$ be a group acting on a hyperbolic space $X$. Assume that the limit set $\Lambda \subset \partial X$ of $G$ contains at least two points. There exist constants $A,B>0$ such that the union of all the $(A,B)$-quasigeodesics between any two distinct points in $\Lambda$ is a quasiconvex subspace on which $G$ acts essentially.
\end{lemma}

\begin{proof}[Sketch of proof.]
Fix two constants $A,B>0$ very large compared to the hyperbolicity constant of $X$ and let $Y \subset X$ denote the union of all the $(A,B)$-quasigeodesics between any two distinct points in $\Lambda$. 

\medskip
Given two points $p,q \in Y$, we claim that any geodesic between $p$ and $q$ stays in a controlled neighbourhood of $Y$. By definition, there exist two $(A,B)$-quasigeodesic lines $\alpha,\beta$ respectively passing through $p,q$ and connecting points $\alpha^\pm, \beta^\pm$ in $\Lambda$. Fix two points $a^\pm \in \alpha$ (resp. $b^\pm$) very far from $p$ (resp. $q$) in such a way that $p$ (resp. $q$) belongs to the subsegment $\alpha_0 \subset \alpha$ between $a^+$ and $a^-$ (resp. $\beta_0 \subset \beta$ between $b^+$ and $b^-$). As a consequence of \cite[Section~6.4]{MR2243589}, there exists a subtree $T \subset X$ containing $p,q,a^\pm, b^\pm$ and roughly embedded (i.e.\ the metric on $T$ differs from the metric induced by $X$ only by an additive constant (independent of our points)). It follows from the Morse property \cite[III.H Theorem 1.7]{bridson_haefliger} that a geodesic in $X$ between $p$ and $q$ must stay close to the arc in $T$ between $p$ and $q$, which is itself in a controlled neighbourhood of $\alpha_0 \cup \beta_0$ and a fortiori of $Y$. 

\medskip 
Thus, we have proved that $Y$ is quasiconvex and hence hyperbolic. Clearly, $G$ preserves $Y$ so we may restrict the action to $Y$. It remains to show the action on $Y$ is essential.  For this, assume that $Z \subset Y$ is a $G$-invariant quasiconvex subspace. Fix a point $z \in Z$. Then $\partial Z$ has to contain the accumulation points of the orbit $G . z$ in $\partial X$, namely the limit set $\Lambda$. Because $Z$ is quasiconvex, any two points in $\partial Z$ are connected by a quasigeodesic, so it follows from the Morse property that $Z$ is quasidense in $Y$. Therefore, $G$ acts on $Y$ essentially. 
\end{proof}

\section{Amenable groups}
\label{s:amenable}

In this section, we focus on Property $\nl$ in the context of \emph{amenable} groups. A group $G$ is amenable if there exists a finitely additive probability measure on $G$ that is invariant under left translation. The structure of amenable groups allows us to prove strong results for this class of groups. A first result in this vein is the following. 

\begin{proposition}\label{prop:Amenable}
Every amenable group $G$ satisfies $\ngt$. Moreover, $G$ satisfies (hereditary) $\nl$ if and only if every homomorphism from (a finite-index subgroup of) $G$ to $\mathbb{R} \rtimes \mathbb{Z}/2\mathbb{Z}$ has image of order at most $2$. If $G$ is finitely generated, this amounts to saying that it does not (virtually) surject onto $\mathbb{Z}$ nor~$\mathbb{D}_\infty$. 
\end{proposition}

\begin{proof}
It follows from the standard Ping-Pong argument, that a group admitting a general type action on a hyperbolic space contains a non-abelian free subgroup. However, any amenable group $G$ contains no free subgroups, and so it satisfies $\ngt$. If $G$ admits a focal action or a lineal action, then we know from Corollary~\ref{cor:LinealAmenable} that there exists a homomorphism $G \to \mathbb{R} \rtimes \mathbb{Z}/2\mathbb{Z}$ (which just takes values in $\mathbb{R}$ if the action is focal or lineal and oriented) whose image has order greater than $2$. Conversely, given a group homomorphism $G \to \mathbb{R} \rtimes \mathbb{Z}/2\mathbb{Z}$ there exists an isometric action on $\mathbb{R}$, simply because $\mathbb{R} \rtimes \mathbb{Z}/2\mathbb{Z} = \mathrm{Isom}(\mathbb{R})$. This action has loxodromic elements if and only if the restriction to the subgroup of elements whose projection to $\mathbb{Z}/2\mathbb{Z}$ is trivial, is unbounded. In turn, this occurs whenever the image of the homomorphism is of order greater than $2$. The statement for hereditary $\nl$ follows since subgroups of amenable groups are amenable. This concludes the proof of the second statement.
 
 Lastly suppose that $G$ is finitely generated. Clearly if $G$ surjects onto $\mathbb{Z}$ or $\mathbb{D}_\infty$, then there exists a homomorphism of $G$ to $\mathbb{R} \rtimes \mathbb{Z}/2\mathbb{Z}$ such that the image has order greater than $2$.
Conversely, let $\psi : G \to \mathbb{R} \rtimes \mathbb{Z}/2\mathbb{Z}$ be such a homomorphism. Let $H$ be the subgroup of elements whose image lands in $\mathbb{R}$, which has index at most $2$. Then $\psi(H)$ is a finitely generated unbounded subgroup of $\mathbb{R}$, which implies that it is isomorphic to $\mathbb{Z}^n$ for some integer $n \geq 1$. If $H = G$, then we can produce a surjective homomorphism $G \to \mathbb{Z}$. Otherwise, let $s \in G$ be such that $\psi(s)$ has order $2$. Then $\psi(G) = \langle \psi(H), \psi(s) \rangle \cong \mathbb{Z}^n \rtimes \mathbb{Z}/2\mathbb{Z}$. Since the action of the involution preserves each factor of $\mathbb{Z}^n$, this in turn surjects onto $\mathbb{Z} \rtimes \mathbb{Z}/2\mathbb{Z} \cong \mathbb{D}_\infty$
\end{proof}

 \subsection{Abelian and nilpotent groups} Abelian and nilpotent groups are widely studied classes of groups among amenable groups. The rigid structure of both these classes of groups allows us to prove even stronger results about them and Property $\nl$.

\begin{proposition}\label{prop:nosubsemi}
Let $G$ be a group with no non-abelian free sub-semigroup. Then $G$ satisfies $\nne$. If $G$ is amenable and admits a lineal action, then the action factors through a homomorphism $G \to \mathbb{R} \rtimes \mathbb{Z}/2\mathbb{Z}$. 
\end{proposition}

\begin{proof}
Let $G$ be a group with no non-abelian free subsemigroups, and let $G$ act on some hyperbolic space $X$. The action cannot be of general type nor focal, as the standard Ping-Pong argument will produce free subgroups or free sub-semigroups, respectively. Therefore $G$ satisfies $\nne$.

For the second statement, let $G$ be an amenable group. First suppose that $G \acts Z$ is an orientable lineal action. Let $\{ \xi_-, \xi_+ \} \subset X$ be the finite orbit at infinity. The corresponding Busemann quasimorphism $\varphi : G \to \R$ is a homomorphism, by Corollary \ref{cor:LinealAmenable}.  By Lemma \ref{lem:constrlin}, there is a generating set $X \subset G$ such that $\varphi : (G, d_X) \to \R$
 is a $G-$equivariant quasi-isometry. Further, it follows from Lemma \ref{lem:constrlin}, that the orbit map $q \colon(G, d_X) \to Z$ defines a $G-$equivariant quasi-isometry. Consequently, there is a $\varphi-$equivariant quasi-isometry between the orbit $Z' = G.z$ and $\R$ (given by $g.z \to \varphi(g)$). As the action of $G$ is lineal, $Z'$ is quasi-convex and $G-$ invariant in $Z$. A similar explanation hols in the non-orientable case, using the Busemann quasicocycle, Proposition \ref{prop:quasicocycle} and Lemma \ref{lem:constrlin2}.
 \end{proof}

\begin{remark} As a direct consequence of the above result, the following groups have $\nne$: Groups of subexponential growth, virtually nilpotent groups \cite{nofreesubsemigroups}  and supramenable groups \cite{supramenable}. 
\end{remark}

\begin{corollary}\label{cor:AbelianNL}
Let $G$ be an abelian group. Every action of $G$ on a hyperbolic space is elliptic, horocyclic, or lineal and oriented. In the latter case, the action factors through a homomorphism $G \to \mathbb{R}$. 
\end{corollary}
\begin{proof}
Every lineal action of an abelian group is necessarily orientable: this follows from Corollary \ref{cor:LinealAmenable} and the fact that the only abelian subgroups of $\mathbb{R} \rtimes \mathbb{Z}/2\mathbb{Z}$ are either of order $2$ or contained in $\mathbb{R}$. The result now follows from Proposition \ref{prop:nosubsemi}.
\end{proof}

\begin{corollary}
An abelian group $G$ is (hereditary) $\nl$ if and only if it is a torsion group. 
\end{corollary} 

\begin{proof}
Let $T$ denote the torsion subgroup of $G$. So $G/T$ is a torsion-free abelian group. As a consequence, $G/T$ embeds into the tensor product $G/T \otimes_\mathbb{Z} \mathbb{Q}$ where $G/T$ is thought of as a $\mathbb{Z}$-module. This tensor product is naturally a vector space over $\mathbb{Q}$, so, as an abelian group, it is a direct sum of copies of $\mathbb{Q}$. It follows that either $G/T$ is trivial, which amounts to saying that $G$ is a torsion group; or $G/T$ (and a fortiori $G$) has a non-trivial homomorphism to $\mathbb{Q}$ (and a fortiori to $\mathbb{R}$). In this case, $G$ cannot be $\nl$, by Corollary \ref{cor:nl:noqm}.

By Corollary~\ref{cor:AbelianNL}, we conclude that an abelian group is $\nl$ if and only if it is a torsion group. As being virtually a torsion group amounts to being a torsion group, the complete statement of the corollary follows. 
\end{proof}

\begin{corollary}
\label{cor:NilpotentNL}
Let $G$ be a nilpotent group. Every action of $G$ on a hyperbolic space is elliptic, horocyclic, or lineal and oriented. In the latter case, the action factors through a homomorphism $G \to \mathbb{R}$. Consequently, a nilpotent group $G$ is $\nl$ if and only if every homomorphism to $\mathbb{R}$ is trivial.
\end{corollary}

\begin{proof}
The fact that every action is elliptic, horocyclic or lineal follows from Proposition \ref{prop:nosubsemi} since nilpotent groups contain no free sub-semigroups; see \cite{nofreesubsemigroups}. So it remains to exclude non-oriented lineal actions.

Suppose that $\varphi : G \to \mathbb{R} \rtimes \mathbb{Z}/2\mathbb{Z}$ is a homomorphism that contains an unbounded subgroup of $\mathbb{R}$. By Corollary \ref{cor:LinealAmenable}, it suffices to show that such a homomorphism necessarily lands in $\mathbb{R}$. Otherwise, there exists an element $g \in \mathbb{R}$ of infinite order, and an element $s$ of order $2$, such that $\langle g, s \rangle \leq \varphi(G)$. But $\langle g, s \rangle \cong \mathbb{D}_\infty$, which is not nilpotent (for instance because it has trivial center). Moreover, it is a subgroup of a quotient of $G$, and this contradicts the fact that $G$ itself is nilpotent.
\end{proof}

\subsection{Solvable groups} Solvable groups are another well studied class of amenable groups, which include abelian and nilpotent groups. However, unlike abelian and nilpotent groups, these can admit (many) focal actions. Indeed, recent results from \cite{BalLamp, ABR, AR} have classified the actions of well studied solvable groups, including Lamplighter and solvable Baumslag Solitar groups, all of which admit focal actions. Thus, we prove a result that classifies which solvable groups admit focal actions. 

\begin{proposition}
A finitely generated solvable group is either virtually nilpotent or contains a finite-index subgroup admitting a focal action. 
\end{proposition}

\begin{proof}
The proposition is essentially a consequence of \cite{MR474887} (see also \cite[Theorem~1.6]{MR2402591}), which states that a finitely generated group that is \emph{just non-virtually nilpotent} (i.e.\ that is not virtually nilpotent but all of whose proper quotients are virtually nilpotent) must be virtually metabelian. Moreover, in case it is metabelian, it has to embed in the affine group $\mathrm{Aff}(k)$ over some local field $k$ with cocompact image. 

Now, let $G$ be a finitely generated solvable group that is not virtually nilpotent. Then $G$ admits a quotient $\overline{G}$ that is just non-virtually nilpotent (see for instance \cite[Claim~2 page~961]{MR2402591}). As a consequence of the result mentioned above, $\overline{G}$ contains a metabelian subgroup $H$ of finite index. Because $\overline{G}$ is not virtually nilpotent, so is $H$, which implies that $H$ also surjects onto some just non-virtually nilpotent group $\overline{H}$. As a quotient of a metabelian group, $\overline{H}$ must be metabelian, so it embeds into $\mathrm{Aff}(k)$ over some local field $k$ with cocompact image.

Thus, $G$ contains a finite-index subgroup (namely, the pre-image of $H$ under $G \twoheadrightarrow \overline{G}$) that surjects onto $\overline{H}$, and the latter admits a focal action either on a real or complex hyperbolic space (if $k=\mathbb{R}$ or $\mathbb{C}$, see \cite[Chapter 2]{martelli}) or on a Bruhat--Tits tree (if $k$ is non-Archimedean, see for instance \cite[Section~II.1.3]{Serre} or \cite[Section~4]{MR1306556}). 
\end{proof}

The following corollary is now an easy consequence.

\begin{corollary}
A finitely generated solvable group is hereditary $(\mathrm{NNE})$ if and only if it is virtually nilpotent. 
\end{corollary}

\section{Stability under operations}
\label{s:stability}

In this section, we study the question of when the properties $\nl, \ngt, \nne$ are preserved under group operations. We start with the simplest ones. 

\begin{proposition}\label{prop:DirectSums} The properties $(\mathrm{NGT})$, $(\mathrm{NNE})$, $(\mathrm{NL})$ as well as their hereditary versions are preserved under taking 
\begin{itemize}[nolistsep,noitemsep]
\item[(i)] Quotients
\item[(ii)] Directed unions 
\item[(iii)] Direct sums
\end{itemize}
\end{proposition}

We will prove each part of the proposition in a series of results. In order to prove stability under quotients, we will make use of the following result which is straightforward to verify. 

\begin{lemma}\label{lem:quotaction} Let $N$ be a normal subgroup of $G$ and $\phi: G \to G/N$ be the quotient map. If $G/N \acts X$ is a hyperbolic action, then the action of $G \acts X$ defined by $g.x = \phi(g).x$ is a well defined hyperbolic action of the same type. In particular, if $G/N \acts X$ contains a loxodromic element,  then so does $G \acts X$. \end{lemma}

\begin{proof}[Proof of Proposition \ref{prop:DirectSums}(i) and (ii)] The proof of (i) follows from Lemma \ref{lem:quotaction} and the fact that finite-index subgroups of $G/N$ are of the form $H/N$, where $H \leq G$ is a finite-index subgroup such that $N \leq H \leq G$. 

The assertions in (ii) follows from the fact that only finitely many elements are needed to determine whether a given action is of general type, non-elementary, or contains a loxodromic (namely, two independent loxodromic isometries, two loxodromic isometries with distinct quasi-axes, one loxodromic isometry); and from the fact that a finite-index subgroup in a directed union of groups is a directed union of finite-index subgroups. 
\end{proof}

In order to deal with direct sums, we will need the following general results. 

\begin{lemma}\label{lem:ProductActingHyp}
Let $G$ be a direct sum $\displaystyle \bigoplus_i H_i$ acting on a hyperbolic space $X$. If there exists some $i$ such that the restriction of the action to $H_{i}$ is not elliptic, then the actions of $G$ and $H_i$ on $X$ have the same type.
\end{lemma}

\begin{proof}
An element $h \in H_i$ is loxodromic for the restricted action $H_i \curvearrowright X$ if and only if its image in $G$ is loxodromic for the action $G \curvearrowright X$. This immediately implies that if the action of $H_i$ has general type, then the action of $G$ has general type.

If the action of $H_i$ is focal, then it fixes a point at infinity $\xi$, and it contains two loxodromic elements $h_1, h_2$ with distinct limit sets at infinity. Since $H_j$ commutes with $H_i$ for all $i \neq j$, the point $\xi$ must be fixed by the other $H_j$ as well, and therefore also by $G$. Moreover $G$ contains $h_1, h_2$, which are two loxodromic elements with distinct limit sets at infinity, and so the action of $G$ is also focal.

If the action of $H_i$ is lineal, then it preserves a quasi-line, which has to be preserved by each $H_j$ for $j \neq i$, and so the action of $G$ is also lineal.

Finally, suppose that the action of $H_i$ is horocyclic. Then it has a unique fixed point $\xi$ at infinity, which is fixed by all of $G$ by commutativity. If there existed a loxodromic element $g \in H_j$ for some $j \neq i$, then $H_i$ would preserve both $ g^{\pm \infty}$, which is not possible by hypothesis. Therefore no $H_j$ can contain a loxodromic element, which implies that the action of $H_j$ is either horocyclic or elliptic. 

Now consider an element $g \in G$, such that $g =(h_j)$ acts as $h_j \in H_j$, where all but finitely many entries are the identity element. If all components $h_j$ are elliptic, then it is straightforward to check that $g$ is elliptic. As $H_j$ contains no loxodromic elements, it suffices to consider the case when some component $h_j$ is parabolic. Then for any $n \in \Z$, $g^n$ contains a parabolic component $h^n_j$, which fixes the unique fixed point $\xi$. By commutativity, all components of $g^n$ fix $\xi$ and therefore so does $g^n$ for every $n$. Consequently, $g$ is a parabolic element. Thus the action of $G$ is horocyclic, as it has an unbounded orbit (since $H_i$ does) and no loxodromic elements.
\end{proof}

The only case the previous lemma does not cover, is when every $H_i$ is elliptic. This is taken care of by the following lemma.

\begin{lemma}\label{lem:CommutingElliptic}
Let $A,B$ be two groups acting on a metric space $X$ with bounded orbits. If every isometry in $A$ commutes with every isometry in $B$, then $\langle A,B \rangle$ has bounded orbits in $X$.
\end{lemma}

\begin{proof}
Fix a point $x \in X$. Let $M$ (resp. $N$) denote the diameter of $A . x$ (resp. $B . x$). Then 
$$d(x,ab . x) \leq d(x,a . x) + d(x,b . x) \leq M+N$$
for every $a \in A$, $b \in B$. But $A$ and $B$ commute, so $\langle A , B \rangle = AB$. We conclude that the orbit of $x$ under $\langle A,B \rangle$ has diameter $\leq M+N$. 
\end{proof}

We are now ready to complete the proof of Proposition \ref{prop:DirectSums}.

\begin{proof}[Proof of Proposition~\ref{prop:DirectSums} (iii).]
Let $A,B$ be two groups. If $A \oplus B$ acts on some hyperbolic space $X$, then it follows from Lemmas~\ref{lem:ProductActingHyp} and~\ref{lem:CommutingElliptic} that the action $A \oplus B \curvearrowright X$ has the same type as $A \curvearrowright X$ or $B \curvearrowright X$. More precisely, $A \oplus B$ is elliptic if and only if $A$ and $B$ are elliptic by Lemma \ref{lem:CommutingElliptic}. If either $A$ or $B$ is not elliptic, then $A \oplus B$ has the same type.
This implies that the properties $(\mathrm{NGT})$, $(\mathrm{NNE})$, and $(\mathrm{NL})$ are preserved by finite direct sums. 

For the hereditary version, let $H \leq A \times B$ be a finite index subgroup. Consider the projections $A', B'$ of $H$ to each factor. As $H$ is finite index in $A \times B$, it follows that $A' \leq A$ and $B' \leq B$ subgroups of finite index and $H$ is a finite index subgroup of $A' \times B'$, with surjective projections to each factor. By Goursat's lemma, there exist subgroups $N \trianglelefteq A' , M \trianglelefteq B'$ and an isomorphism $\theta : A'/N \to B'/M$ such that $H$ can be identified with the graph of the map $\theta$. As $A,B$ are hereditary $\nl$, so are $A', B'$. As the property survives under quotients, $A'/N, B'/M$ are also hereditary $\nl$. As $H$ is the graph of the isomorphism $\theta$, it follows that $H$ is also hereditary $\nl$. 

Thus, the proposition holds for finite direct sums. The general case now follows from Proposition \ref{prop:DirectSums} (ii) about stability under directed unions.
\end{proof}

It is worth noting that only the hereditary versions of the properties are preserved under commensurability. Recall that two groups $G, H$ are \emph{commensurable} if they contain isomorphic finite-index subgroups.

\begin{lemma}
\label{lem:fi:subgroup}

Let $G_0 \leq G$ be a subgroup of finite-index. If $G_0$ is $\nl$ (respectively $\nne, \ngt$), then $G$ is $\nl$ (respectively $\nne, \ngt$).
\end{lemma}

\begin{proof}
This follows easily from the fact that, if $g \in G$ is loxodromic for an action on a hyperbolic space $X$, then $g^{[G : G_0]} \in G_0$ is loxodromic with the same axis.
\end{proof}

\begin{proposition}
Property hereditary $\nl$ (respectively, hereditary $\nne$, hereditary $\ngt$) is stable under commensurability.
\end{proposition}

\begin{proof}
Let $G, H$ be commensurable groups, so there exist finite-index subgroups $G_0 \leq G$ and $H_0 \leq H$ that are isomorphic. Suppose that $G$ is hereditary $\nl$. Then $G_0$ is $\nl$, and thus $H_0$ is also $\nl$. It then follows from Lemma \ref{lem:fi:subgroup} that $H$ is also $\nl$. The same proof works for the other hereditary properties in the statement.
\end{proof}

The following example shows that, in contrast, Property $\nl$ need not be preserved under commensurability. We will go into much more detail about permutational wreath products in Section \ref{s:wreath}, and this example may be considered as a warm-up.

\begin{example}\textit{A group which is $\nl$ but not hereditary $\nl$}. 
Let $G = (\mathbb{D}_\infty \oplus \mathbb{D}_\infty) \rtimes (\Z / 2\Z)$, where the generator $t$ of $\Z / 2\Z$ swaps the factors of the direct sum. Obviously, $H = \mathbb{D}_\infty \oplus \mathbb{D}_\infty$ has finite index in $G$ and it fails to have $\nl$. This implies that $G$ is not hereditary $\nl$. However, it is $\nl$. By Proposition~\ref{prop:Amenable}, since $G$ is finitely generated and amenable, it suffices to show that $G$ does not surject onto $\mathbb{Z}$ nor $\mathbb{D}_\infty$. We will prove that every homomorphism $G \to \mathbb{D}_\infty$ has finite image, which implies both statements at once.

We start by noticing that the only elements of order $2$ in $\mathbb{D}_\infty$ are the reflections, and two reflections commute if and only if they are equal. Now let $\varphi : G \to \mathbb{D}_\infty$ be a homomorphism. A presentation for the group $G$ is as follows 
\begin{align*}
   G =  \langle a, b, x, y, t \mid & a^2 = b^2 = x^2 = y^2 = t^2 = 1, \\
    & [a, x] = [b, x] = [a, y] = [b, y] = 1, \\
    & tat^{-1} = x, tbt^{-1} = y \rangle.
\end{align*}
If $\varphi(a) = 1$, then $\varphi(x) = 1$ as well, so $\varphi$ factors through $\langle G \mid a, x \rangle \cong \mathbb{Z}/2\mathbb{Z} \wr \mathbb{Z}/2\mathbb{Z}$, which is finite. A similar conclusion holds if any of $\varphi(b), \varphi(x), \varphi(y)$ equal $1$. Therefore we may assume that none of $a, b, x, y$ lies in the kernel of $\varphi$. By the previous remark, it follows that $\varphi(a), \varphi(x)$ are commuting reflections, therefore $\varphi(a) = \varphi(x)$. Repeating the argument with $a, y$ and then with $b, y$, we obtain that each of $a, b, x, y$ have the same image. Thus $\varphi$ factors through $\langle G \mid a = b = x = y \rangle \cong \mathbb{Z}/2\mathbb{Z} \times \mathbb{Z}/2\mathbb{Z}$, which is also finite.
\end{example}

The following example shows that $\nne$ and $\ngt$ are not preserved under commensurabilty. Again, the reader should consider this as a warm-up to Section \ref{s:wreath}.

\begin{example}\label{ex:nlnothere} \textit{A group which is $(\mathrm{NNE})$ but not hereditary $(\mathrm{NGT})$}.
Let $F$ be a non-abelian free group. It is easy to see that the group $G:= (F \oplus F) \rtimes \mathbb{Z}/2\mathbb{Z}$ is not hereditary $(\mathrm{NGT})$. Assume that $G$ acts on a hyperbolic space $X$. If the left copy of $F$ in $G$ contains two loxodromic elements $a$ and $b$, then so does the right copy of $F$ because the two copies of $F$ are conjugate. Thus, $a$ and $b$ commute with a common loxodromic isometry, which implies that they share the same quasi-axis. Thus, the actions of the two copies of $F$ must be elementary. It follows from Lemma~\ref{lem:ProductActingHyp} that $F \oplus F \curvearrowright X$ is elementary, and therefore so is $G \curvearrowright X$ (see the proof of Lemma \ref{lem:fi:subgroup}). Hence $G$ is $\nne$.
\end{example}

\subsection{Stability under extensions} In this section, we shall study the stability of the properties $\ngt, \nne, \nl$ and their hereditary versions under extensions. The main result of this section is the following. 

\begin{proposition}\label{prop:extension}
Let $G$ be a group that fits in a short exact sequence 
$$1 \to N \to G \to Q \to 1.$$
If $N$ and $Q$ are both (hereditary) $\nl$, then so is $G$.
\end{proposition}

In order to prove the above result, we need the following lemmas about factoring actions through quotients. 

\begin{lemma}\label{lem:NormalElliptic}
Let $G$ be a group, $N \lhd G$ a normal subgroup, and $X$ a hyperbolic space on which $G$ acts. If $N$ has bounded orbits, then $G \curvearrowright X$ factors through $G \twoheadrightarrow G/N$. \end{lemma}

\begin{proof}
Up to replacing $X$ with the graph whose vertex-set is $X$ and whose edges connect two points $x,y$ of $X$ whenever $d(x,y) \leq 1$, we can assume without loss of generality that $X$ is a graph. Fix a constant of hyperbolicity $\delta \geq 0$ of $X$. Let $F$ denote the set of all the vertices whose $N$-orbits have diameters $\leq 2 \delta$. We claim that $F$ is non-empty and $12\delta$-quasiconvex. 

\medskip \noindent
The fact that $F$ is non-empty is given by \cite[Lemma~2.1]{MR1776048}. More precisely, given a bounded subset $S \subset X$ and a vertex $x \in X$, define $r_x:= \min \{ r \geq 0 \mid S \subset B(x,r) \}$. A center of $S$ is a vertex $x \in X$ satisfying $r_x= \min \{ r_z \mid z \in X\}$. According to \cite[Lemma~2.1]{MR1776048}, the set of all the centers of $S$ has diameter $\leq 2\delta$. (It is assumed in \cite[Lemma~2.1]{MR1776048} that the hyperbolic space is proper. But this assumption is only required for the existence of centers, which is clear in our case as $X$ is a graph.) Thus, given an arbitrary $N$-orbit, which is bounded by assumption, it is easy to check that the orbit of a center must have diameter $\leq 2\delta$. Hence $F$ is non-empty, as this center will belong to $F$. 

\medskip \noindent
Next, to see that $F$ is $12\delta$-quasiconvex, fix two vertices $x,y \in F$ and a geodesic $[x,y]$ connecting $x$ to $y$. We know from \cite[Corollary~10.5.3]{MR1075994} that the metric of $X$ is $8\delta$-convex, i.e.\ for any two constant-speed parametrizations $\gamma_1,\gamma_2 : [0,1] \to X$ of two geodesics,
$$d \left( \gamma_1(ta + (1-t)b), \gamma_2(ta+(1-t)b) \right) \leq t d(\gamma_1(a),\gamma_2(a)) + (1-t) d(\gamma_1(b),\gamma_2(b))+ 8\delta$$
for all $a,b \in [0,1]$ and $t \in [0,1]$. As a consequence, for every $z \in [x,y]$, we have 
$$d(z,nz) \leq d(x,nx) + d(y,ny) + 8\delta \leq 12 \delta$$
for every $n \in N$. In other words, the $N$-orbit of $z$ has diameter $\leq 12\delta$. Fixing a center $c$ of $N \cdot z$, we have $c \in F$ and consequently $d(c,z) \leq 12 \delta$. Thus, $[x,y]$ lies in the $12\delta$-neighbourhood of $F$, proving our claim.

\medskip \noindent
Now, let $Y$ denote the graph obtained from $X$ by first adding a new vertex $x_S$, an \emph{apex}, for every subset $S \subset X$ of diameter $\leq 2\delta$, which we connect to all the vertices in $S$, and then by connecting two apices $x_R,x_S$ with an edge whenever $d(R,S) \leq 12\delta$. Clearly, the inclusion map $X \hookrightarrow Y$ is a quasi-isometry that sends $F$ at finite Hausdorff distance from the subgraph $Z \subset Y$ induced by all the apices fixed by $N$. Because $F$ is $G$-invariant (as $N$ is normal), non-empty and quasi-convex, it is quasi-dense in $X$. It follows that $Z$ is quasi-dense in $Y$, and therefore it is a non-empty, connected, hyperbolic graph on which $G$ naturally acts. Moreover, by construction, $N$ lies in the kernel of this action. This proves that $G \curvearrowright X$ factors through $G \twoheadrightarrow G/N$. 
\end{proof}

\begin{corollary}\label{cor:ShortExactSequence}
Let $G$ be a group and $N \lhd G$ a normal subgroup. If $N$ is $\ngt$, then every action of general type of $G$ on some hyperbolic space factors through $G \twoheadrightarrow G/N$.
\end{corollary}

\begin{proof}
If the restricted action $N \curvearrowright X$ is horocyclic, lineal, or focal, then $N$ fixes one or two points in $\partial X$, which are preserved by $G$ since $N$ is normal. In particular, in this case $G \curvearrowright X$ cannot be of general type. Consequently, $N \curvearrowright X$ must be elliptic and the desired conclusion follows from Lemma~\ref{lem:NormalElliptic}.
\end{proof}

\begin{proof}[Proof of Proposition~\ref{prop:extension}.]
Let $G$ act on some hyperbolic space $X$. Up to replacing $X$ with the quasiconvex hull of the limit set of $G$, we assume that the action is essential. Because $N$ has $\nl$, the induced action $N \curvearrowright X$ is either elliptic or horocyclic. In the former case, we know from Lemma~\ref{lem:NormalElliptic} that $G \curvearrowright X$ factors through $Q$, and we can conclude that $G \curvearrowright X$ has no loxodromic elements since $Q$ has $\nl$. So we may assume that $N \curvearrowright X$ is horocyclic.

\medskip 
The unique point at infinity fixed by $N$ has to be fixed by $G$ as well, so $G \curvearrowright X$ is horocyclic, lineal, or focal. In the first case, $G$ has no loxodromic. The second case is impossible because there is no horocyclic action on a quasi-line. In the third case, the Busemann quasimorphism (Proposition \ref{prop:busemann}) is an unbounded quasimorphism $G \to \mathbb{R}$. The restriction to $N$ is bounded, because $N$ has $\nl$, and therefore it has no unbounded quasimorphisms (Corollary \ref{cor:nl:noqm}). So the Busemann quasimorphism descends to an unbounded quasimorphism $Q \to \mathbb{R}$ \cite[Remark 2.90]{scl}, which is impossible by Corollary \ref{cor:nl:noqm}, since $Q$ has $\nl$. Thus, we have proved that $G \curvearrowright X$ cannot contain a loxodromic element, as desired.

\medskip 
Next, if $N$ and $Q$ have hereditary $\nl$ and if $H \leq G$ is a finite-index subgroup, then there exist finite-index subgroups $\dot{N} \leq N$ and $\dot{Q} \leq Q$ such that $H$ fits into a short exact sequence
$$1 \to \dot{N} \to H \to \dot{Q} \to 1.$$
Because $\dot{N}$ and $\dot{Q}$ have $\nl$, it follows from what we have just proved that $H$ has $\nl$. Thus, $G$ is hereditary $\nl$.
\end{proof}

\subsection{Central extensions} In addition to the situation considered above, some of the properties are also preserved under central extensions, which we prove here. 

\begin{lemma}
\label{lem:factorsthroughcenter}

Let $G$ be a group. Every non-elementary action of $G$ on a hyperbolic space factors through $G \twoheadrightarrow G/Z(G)$.  
\end{lemma}

\begin{proof}
By Lemma~\ref{lem:Essential}, up to replacing $X$ with a $G$-invariant quasiconvex subspace we can assume that $G \curvearrowright X$ is essential. Let $\xi \in \Lambda(G)$ be a point in the limit set of $G$. We can write $\xi$ as the limit of $(g_n . x)_{n \in \mathbb{N}}$ for some basepoint $x \in X$ and some sequence $(g_n)_{n \in \mathbb{N}}$ of elements of $G$. For every $z \in Z(G)$, we have
$$z . \xi = z . \lim\limits_{n \to \infty} g_n . x = \lim\limits_{n \to \infty} g_n . zx = \xi.$$
Thus, $Z(G)$ fixes pointwise $\Lambda(G)$. Since $\Lambda(G)$ is infinite and the action of $G$ essential, it follows that there exists some $K \geq 0$ such that $Z(G)$ moves every point of $X$ at distance at most $K$. In particular, $Z(G)$ is elliptic, so the desired conclusion follows from Lemma~\ref{lem:NormalElliptic}. 
\end{proof}

\begin{corollary}\label{cor:ModCenter}
A group $G$ is $\ngt$ (resp. $\nne$) if and only if $G/Z(G)$ is $\ngt$ (resp. $\nne$).
\end{corollary}

\subsection{Permutational wreath products}
\label{s:wreath}

In this section, we study the stability of Property $\nl$ under taking permutational wreath products. This can be precisely characterized by the following theorem.

\begin{theorem}\label{thm:WreathNL}
Let $B$ be a group acting on a set $S$ and $\{ A_s \}_{s \in S}$ a $B$-invariant family of groups. Let $G$ be the semidirect product $\bigoplus_{s \in S} A_s \rtimes B$ where $B$ acts on the direct sum by permuting the factors according to its action on $S$. Then $G$ is $\nl$ if and only if all of the following conditions are satisfied:
\begin{itemize}[noitemsep]
	\item[(1)] $B$ is $\nl$;
	\item[(2)] for every $s \in S$ fixed by $B$, $A_s$ is $\nl$;
	\item[(3)] for every $s \in S$ with a finite $B$-orbit, $A_s$ has no unbounded quasimorphism;
	\item[(4)] for every $s \in S$ with an infinite $B$-orbit, $A_s$ has no unbounded homomorphism to $\mathbb{R}$.
\end{itemize}
\end{theorem}

Before turning to the proof of the theorem, we record the particular case of permutational wreath products.

\begin{corollary}\label{cor:WreathNL}
Let $A,B$ be two groups and $S$ a set on which $B$ acts. The wreath product $A \wr_S B$ is $\nl$ if and only if $B$ is $\nl$ and one of the following conditions hold:
\begin{itemize}[noitemsep]
	\item[1.] $S$ contains a point fixed by $B$, and $A$ is $\nl$;
	\item[2.] No point in $S$ is fixed by $B$ but $S$ contains a finite $B$-orbit, and $A$ has no unbounded quasimorphism;
	\item[3.] All the $B$-orbits in $S$ are infinite, and $A$ has no unbounded homomorphism to $\mathbb{R}$.
\end{itemize}
\end{corollary}

Notice that the third item covers standard wreath products, when $B$ is infinite. As a first step towards the proof of Theorem~\ref{thm:WreathNL}, we focus on the weaker property $\nne$. 

\begin{proposition}\label{prop:WreathNNE}
Let $B$ be a group acting on a set $S$ and $\{ A_s\}_{s \in S}$ a $B$-invariant family of groups. Let $G$ be the semidirect product $\bigoplus_{s \in S} A_s \rtimes B$ where $B$ acts on the direct sum by permuting the factors according to its action on $S$. If $B$ is $\nl$ and $A_s$ is $\nne$ for every $s \in S$ fixed by $B$, then $G$ is $\nne$.
\end{proposition}

\begin{proof}
 Assume that $B$ has $\nl$ and that $A_s$ for $s \in S$ fixed by $B$ has $\nne$. Let $G$ act on some hyperbolic space $X$. Without loss of generality, we assume that $G$ acts essentially on $X$, by Lemma \ref{lem:Essential}. Indeed, if the action of $G$ is elliptic or horocyclic, then we are done. In all other cases, $G$ admits at least 2 limit points on $\partial X$, and so we may apply the lemma to construct an essential action. 

\medskip 
We consider the restriction of this action to the normal subgroup $\bigoplus_s A_s$. If $\bigoplus_s A_s \curvearrowright X$ is elliptic, then we know from Lemma~\ref{lem:NormalElliptic} that $G \curvearrowright X$ factors through $B$. Consequently, $G \curvearrowright X$ must be elementary. If $\bigoplus_s A_s \curvearrowright X$ is horocyclic or lineal, then the finite set stabilized at infinity must be stabilized by $G$ as well. In the lineal case, $G \curvearrowright X$ must also act lineally, as it fixes 2 distinct points on the boundary. In the horocyclic case, the action of $G$ can be horocyclic or focal, as $G$ fixes a unique point on the boundary. In the former case, we are done. In the latter case, $B$ must admit a loxodromic element for the action action on $X$, which contradicts $\nl$.

\medskip 
It remains to consider the case when $\bigoplus_s A_s \curvearrowright X$ is non-elementary. We deduce from Lemma~\ref{lem:ProductActingHyp} that there exists some $s \in S$ such that $A_s \curvearrowright X$ is non-elementary. Necessarily, $A_s$ contains two loxodromic isometries, say $g$ and $h$, with an infinite Hausdorff distance between their quasi-axes. If $s$ is not fixed by $B$, say $b . s \neq s$ for some $b \in B$, then $bhb^{-1}$ is a loxodromic isometry commuting with both $g$ and $h$, which implies that $g$ and $h$ must share the same quasi-axis, namely the quasi-axis of $bhb^{-1}$. This contradicts our assumption about $g$ and $h$, so $s$ has to be fixed by $B$. But this contradicts the assumption that for such an $s \in S$, $A_s \curvearrowright X$ must be elementary.
\end{proof}

\begin{remark}[Partial converse] Note that if $s \in S$ is fixed by $B$, then $B$ and $A_s$ are both quotients of $G$. By Proposition \ref{prop:DirectSums} they must have $\nne$ if $G$ does. We may consider the following question to consider the converse of Proposition \ref{prop:WreathNNE}:  if $A$ is non-trivial and if $B$ admits a lineal action, does $A \wr B$ admits a focal action? While a positive answer is reasonable to expect, providing a complete proof is beyond the scope of this paper. 
\end{remark}

We record the following particular case of Proposition~\ref{prop:WreathNNE}, before proceeding to the proof of Theorem \ref{thm:WreathNL}.

\begin{corollary}
Let $A,B$ be two groups and $S$ a set on which $B$ acts. If no point of $S$ is fixed by $B$, then $A \wr_S B$ is $\nne$ if $B$ is $\nl$. 
\end{corollary}

\begin{proof}[Proof of Theorem~\ref{thm:WreathNL}.]
Suppose $G$ is $\nl$. If $s \in S$ is fixed by $B$, then $B$ and $A_s$ are quotients of $G$. Proposition \ref{prop:DirectSums} implies that they are $\nl$, proving that conditions (1) and (2) hold. If there exists an index $r \in S$ and a map $\varphi : A_r \to \mathbb{R}$ which is either an unbounded quasimorphism in case $r$ has a finite $B$-orbit or an unbounded homomorphism in case $r$ has an infinite $B$-orbit, then 
$$((a_s)_{s \in S}, b) \mapsto \sum\limits_{s \in B. r} \varphi(a_s)$$
defines an unbounded quasimorphism or homomorphism $G \to \mathbb{R}$ respectively. Therefore, by Corollary \ref{cor:nl:noqm}, conditions (3) and (4) must hold as well.

\medskip 
Conversely, assume that the four conditions in the statement of Theorem \ref{thm:WreathNL} are satisfied. Let $X$ be a hyperbolic space on which $G$ acts. We know from Proposition~\ref{prop:WreathNNE} that $G \curvearrowright X$ must be elementary. To prove that $G$ has $\nl$, we only need to rule out the lineal case. Assume by contradiction that $G \curvearrowright X$ is lineal. The induced action of the normal subgroup $\bigoplus_s A_s \curvearrowright X$ is either elliptic or lineal. (Indeed, it cannot be horocyclic as that would imply $G$ fixes a unique point of $\partial X$). In the former case, the action factors through $B$ according to Lemma~\ref{lem:NormalElliptic} and we conclude that $G \curvearrowright X$ has no loxodromic element as $B$ has $\nl$. So we may assume that $\bigoplus_s A_s \curvearrowright X$ is lineal for the remainder of the proof. By Lemma~\ref{lem:ProductActingHyp}, there must exist some $r \in S$ such that $A_r \curvearrowright X$ is lineal. Observe that, as a consequence of our assumption, $B . r$ must be infinite.

\medskip 
Because $G \curvearrowright X$ is lineal and the restriction to $A_r$ is non-elliptic, there exists a quasimorphism $\varphi : H \to \mathbb{R}$ for some subgroup $H \leq G$ of index at most $2$ that is unbounded on $H \cap A_r$. The only case where $H$ needs to be a proper subgroup is when $G$ stabilizes a quasi-line and inverts its endpoints at infinity. In this case, we claim that $A_r$ lies in $H$. i.e.\ it does not invert the two points at infinity. Indeed, if $R \in A_r$ is such a reflection, then there exists a constant $D \geq 0$ (controlled by the hyperbolicity constant of $X$) such that $\{x \in X \mid d(x,Rx) \leq D\}$ is non-empty and bounded. This set has to be stabilized by $A_s$ for every $s \in B . r \backslash \{r\}$, so $A_s$ must be elliptic, contradicting the fact that $A_r \curvearrowright X$ is lineal. Thus, we know that $A_r$ always lies in $H$. 

\medskip 
Up to replacing $H$ with one of its subgroups (namely, $\langle B \cap H, A_s \cap H \ (s \in S) \rangle$), we can write $H$ as $\bigoplus_{s \in S} A_s' \rtimes B'$ for some subgroups $A_s' \leq A_s$ and $B' \leq B$ of indices at most $2$. From now on, we think of $\varphi$ as defined on the subgroup $\langle A_r,B' \rangle = \bigoplus_{s \in B' . r} A_s \rtimes B'$. (Note that since $A_r \subset H$ as proved in the previous paragraph, we have $A_r=A'_r$.) Up to replacing $\varphi$ with a quasimorphism at finite distance, we assume that $\varphi$ is homogeneous and zero on $B'$. As a consequence of \cite[Section 2.2]{scl},
\begin{itemize}
	\item[1.] $\varphi$ is conjugacy-invariant;
	\item[2.] $\varphi(ab)=\varphi(a)+\varphi(b)$ whenever $a$ and $b$ commute;
	\item[3.] its defect is $D(\varphi)= \mathrm{sup}_{a,b} | \varphi ( [a,b] ) |$.
\end{itemize}
It follows from the first two properties that there exists a homogeneous quasimorphism $\psi : A_r \to \mathbb{R}$ such that the restriction $\rho$ of $\varphi$ on $\bigoplus_s A_s$ is given by
$$\rho : (a_s)_{s \in B' . r} \mapsto \sum\limits_{s \in B' . r} \psi(a_s).$$
According to the third property, there exists $a,b \in A_r$ such that $\psi([a,b]) \geq D(\psi)/2$. Fix a finite subset $O \subset B'.r$ of the orbit. As the collection of the $A_s$ is $B$-invariant, the $A_s$ with $s \in O$ are all copies of the same group. By an abuse of notation,  we identify all these groups and define: 
$$\alpha : s \mapsto \left\{ \begin{array}{cl} a & \text{if } s \in O \\ 1 & \text{otherwise} \end{array} \right. \text{ and } \beta : s \mapsto \left\{ \begin{array}{cl} b & \text{if } s \in O \\ 1 & \text{otherwise} \end{array} \right.;$$
which we see as tuples in $\bigoplus_s A_s$. Then 
$$D(\rho) \geq \rho([\alpha,\beta]) = \sum\limits_{s \in O} \psi([a,b]) \geq |O| D(\psi)/2.$$
Because $O$ can be chosen arbitrarily large, as $B' . r$ is infinite, we must have $D(\psi)=0$. In other words, $A_r$ admits an unbounded homomorphism to $\mathbb{R}$, which is a contradiction. 
\end{proof}

Theorem \ref{thm:WreathNL} allows us to prove a weaker version of Corollary \ref{cor:QI-embed}. One can think of this result as evidence of existence of many diverse groups with Property $\nl$.

\begin{corollary}
\label{cor:qi}

Every finitely generated group quasi-isometrically embeds into a finitely generated hereditary $\nl$ group.
\end{corollary}

\begin{proof} 
Let $G$ be a finitely generated group. First, quasi-isometrically embed $G$ into a finitely generated perfect group $G^+$, for instance one can choose $G^+$ to be a twisted Brin--Thompson group \cite{twistedV, twistedV2}. Next, fix a finitely generated infinite group $H$ which is hereditary $\nl$, for instance a finitely generated infinite torsion group. Then $G$ quasi-isometrically embeds into the standard wreath product $G^+ \wr H$.

We claim that $G^+ \wr H$ is hereditary $\nl$. Since $G^+$ is perfect and $H$ is infinite, it follows from \cite{gruenberg} that every finite quotient of $G^+ \wr H$ factors through the quotient $G^+ \wr H \to H$. Since every finite-index subgroup contains a finite-index normal subgroup, it follows that every finite-index subgroup of $G^+ \wr H$ contains the base group $\oplus G^+$. Therefore, for every finite-index subgroup $K < G^+ \wr H$ there exists a finite-index subgroup $H_1 < H$ such that $K$ is the preimage of $H_1$ under the quotient $G^+ \wr H \to H$. Thus we can explicitly describe $K$ as the permutational wreath product $G^+ \wr_H H_1$, where $H_1$ acts on $H$ by left multiplication. Since $H$ is infinite, and $H_1 < H$ has finite index, all orbits for this action are infinite. Finally, since $G^+$ is perfect, it has no unbounded homomorphism to $\mathbb{R}$. Therefore Corollary \ref{cor:WreathNL}(3) applies, and $K$ is $\nl$.
\end{proof}

\begin{remark}
Note that the first step of the proof above consists in embedding $G$ into a twisted Brin--Thompson group $G^+$. We will see later that this is already enough to prove the above result, as twisted Brin--Thompson groups are simple and we will show that they have $\nl$ in Corollary \ref{cor:twistedV}.
\end{remark}

Along the same lines as Theorem~\ref{thm:WreathNL}, one can show that various groups acting on rooted trees satisfy $\nne$ because they naturally decompose as permutational wreath products. In order to make this assertion precise, we need to introduce some vocabulary.

Let $G$ be a group acting on a rooted tree $(T,o)$. For every $n \geq 1$, the \emph{$n$th level} of $T$ refers to the vertices at distance $n$ from $o$. Notice that $G$ stabilizes each level of $T$ since it fixes the root $o$. The $G$-stabilizer of the $n$th level is denoted by $\mathrm{st}_G(n)$. Given a vertex $v \in T$, the \emph{rigid stabilizer} $\mathrm{rig}_G(v)$ of $v$ is the subgroup of $G$ which permutes the vertices \emph{below} $v$ (i.e.\ separated from $o$ by $v$) and fixes all the other vertices of $T$. For every $n \geq 1$, we denote by $\mathrm{rig}_G(n)$ the subgroup generated by the rigid stabilizers of all the vertices in the $n$th level; notice that $\mathrm{rig}_G(n)$ decomposes as the product of these rigid stabilizers. If, for every $n \geq 1$, $G$ acts transitively on the $n$th level of $T$ and $\mathrm{rig}_G(n)$ has finite index in $G$, then $G$ is called a \emph{branch group}. 

We will prove the following result, which will yield implications for branch groups.

\begin{proposition}\label{prop:Branch}
Let $G$ be a group acting on a rooted tree $(T,o)$. If there exists some $n \geq 1$ such that $G$ acts transitively on the $n$th level of $T$ and $\mathrm{rig}_G(n)$ has finite index in $G$, then $G$ satisfies $\nne$. 
\end{proposition}

We defer the proof of the proposition for a bit as it will be an immediate consequence of our next result, which is proved following the lines of Proposition~\ref{prop:WreathNNE}.

\begin{lemma}\label{lem:CommutingSubgroups}
Let $G$ be a group. Assume that $G$ contains pairwise commuting and pairwise conjugate subgroups $A_1, \ldots, A_r, r > 1$ generating a finite-index subgroup of $G$, then $G$ satisfies $\nne$.
\end{lemma}

\begin{proof}
Let $G$ act on a hyperbolic space $X$. Because the $A_i$ are pairwise conjugate, their induced action on $X$ have the same type.

If $A_1$ contains two loxodromic elements $a$ and $b$, then $A_2$ contains a conjugate of $a$ commuting with both $a$ and $b$. Because two commuting loxodromic isometries must have the same quasi-axes up to finite Hausdorff distance, it follows that the quasi-axes of this conjugate, $a$, and $b$ all coarsely coincide. Thus, the $A_i$ must all act elementarily on $X$.

If all the $A_i$ are elliptic, it follows from Lemma~\ref{lem:CommutingElliptic} that $G$ is elliptic. If all the $A_i$ are lineal, then they must stabilize the same quasi-line because they pairwise commute, which implies that the action of $G$ is also lineal. Finally, if all the $A_i$ are horocyclic, then their elements are all elliptic or parabolic and they all fix a common point at infinity. 

This implies that the action of $G$ must be horocyclic as well. Indeed, $G$ must fix the same unique point on $\partial X$ that is fixed by all the $A_i$. Further, if $G$ contains a loxodromic element, then so does the group generated by $A_1, \ldots, A_r$, as it has finite index in $G$. However, as the $A_i$ commute, $\langle A_1, \ldots, A_r \rangle = A_1\ldots A_r$. Thus at least one of the $A_i$'s must contain a loxodromic element (see the argument at the end of the proof of Lemma \ref{lem:ProductActingHyp}), which is a contradiction. So $G$ has no loxodromic elements. \end{proof}

\begin{proof}[Proof of Proposition~\ref{prop:Branch}.]
The desired conclusion follows by applying Lemma~\ref{lem:CommutingSubgroups} to the rigid stabilizers of the vertices from the $n$th level of $T$. 
\end{proof}

We now prove the following corollary that additionally deals with branch groups with the \emph{congruence subgroup property}. A group acting on a rooted tree satisfies the congruence subgroup property if every finite-index subgroup contains the stabilizer of some level. 

\begin{corollary}
\label{cor:Branch}
Branch groups are $\nne$. Branch groups satisfying the congruence subgroup property are hereditary $\nne$.
\end{corollary}

\begin{proof} The result is an immediate consequence Proposition~\ref{prop:Branch}. Indeed, if $K$ is finite index in $G$, then it is contains $\mathrm{st}_G(n)$ for some level $n$ of the tree. For this $n$, $\mathrm{rig}_G(n)$ has finite index in $G$ and hence in $K$, since it is a subgroup of $K$. As $K$ also acts transitively on the $n$-th level (as $G$ does so and $K$ setwise stabilizes the $n$-th level), we can apply Proposition \ref{prop:Branch} to $K$ and conclude that $K$ has $\nne$.
\end{proof}

It is worth noticing that, even though many branch groups are famously torsion groups (e.g.\ the Grigorchuk groups and the Gupta--Sidki groups) and consequently obviously satisfy $\nne$ (and even hereditary $\nl$), there exist branch groups containing non-abelian free subgroups \cite{MR1995624}, for which Property $\nne$ is not so clear at the first glance.

\section{A dynamical criterion for groups of homeomorphisms}\label{thompson}

The goal of this section is to provide a dynamical criterion for a group acting on a compact Hausdorff space by homeomorphisms to have Property $\ngt$. This will then be applied to a variety of Thompson-like groups to show that they are hereditary $\nl$. We warn the reader that we consider compact Hausdorff spaces that are not endowed with a metric, so the actions are merely by homeomorphisms, while actions on hyperbolic spaces are still assumed to be isometric.

Before stating the main result of this section, we recall some basic terminology that will be used in this section.
Let $G \acts X$. The action is \emph{faithful} if $gx = x$ for all $x \in X$ implies that $g$ is the identity. Equivalently, for each non-identity element $g$, there is an $x \in X$ such that $gx \neq x$.
A group $G$ is said to be \emph{boundedly generated by $A \subset G$} if there exists an integer $n \in \N$ such that every $g \in G$ can be expressed as a product of at most $n$ elements in $A$.

Suppose a group $G$ acts faithfully on a compact Hausdorff space $X$. Given $g \in G$, the \emph{support} of $g$ is the set of elements not fixed by $g$, i.e. $\op{supp}(g) = \{x \in X \mid gx \neq x\}$. Note that two elements $g, h \in G$ commute whenever $\op{supp}(g) \cap \op{supp}(h) = \emptyset$, although this is not necessary (take for instance $g = h$).  Throughout this section, by $g^h$, we denote the conjugate $hgh^{-1}$ of $g$, and by $[g, h]$ the commutator $ghg^{-1}h^{-1}$.

Given a collection $\mathcal{I}$ of subsets of $X$, we denote by $\mathcal{I}^{(n)}$ the set of ordered $n$-tuples of elements of $\mathcal{I}$ with pairwise disjoint closures whose union is not dense in $X$. Our criterion will involve transitivity on $\mathcal{I}^{(n)}$; therefore the added condition that the union is not dense will be necessary for this to ever be satisfied, as $G$ sends dense sets to dense sets. Our main result in this section is the following. 

\begin{theorem}
\label{thm:cH:criterion}

Let $G$ be a discrete group acting faithfully on a compact Hausdorff space $X$. Suppose that there is a basis $\mathcal{I}$ of non-dense and non-empty open sets in $X$, such that the following holds:
\begin{itemize}
	\item[1.] \emph{($\mathbf{C}$: Complements)} For every $I \in \mathcal{I}$ there exists $J \in \mathcal{I}$ such that $I^c \subset J$.
    \item[2.] \emph{($2 \mathbf{T}$: Double transitivity)} $G$ preserves $\mathcal{I}$ and acts \emph{doubly transitively} on it: for every two pairs $(I, J), (I', J') \in \mathcal{I}^{(2)}$, there exists $g \in G$ such that $(gI, gJ) = (I', J')$.
    \item[3.] \emph{($3 \mathbf{T}$: Weak triple transitivity)} For every $g, h \in G$ there exist $(M, N, P) \in \mathcal{I}^{(3)}$ such that $(M, N, P)$ and $(g M, h N, P)$ are in the same $G$-orbit in $\mathcal{I}^{(3)}$.
    \item[4.] \emph{($\mathbf{L}$: Local action)} Let $(I, J, K) \in \mathcal{I}^{(3)}$, and let $g, h \in G$ be such that $g I = I$ and $h J = J$. Then there exists $b \in G$ such that $b|_I = g|_I$ and $b|_J = h|_J$ and $b|_K = id|_K$.
\end{itemize}
Then $G$ has Property $\ngt$.
\end{theorem}

Condition ($3 \mathbf{T}$) is strictly weaker than $3$-transitivity: we will see in Corollary \ref{cor:circle:criterion:firststep} that it is also satisfied by some groups acting on the circle. On the other hand it is easy to see that having actual $3$-transitivity on the circle prevents the action from being orientation-preserving. Some condition stronger than double transitivity is needed, as the following example shows (we will go through similar examples in Section \ref{s:cantor}).

\begin{example}
Let $G$ be the group of isometries of a locally finite regular tree $T$, and let $X := \partial T$ be the boundary at infinity of $T$. Then $G$ acts faithfully by homeomorphisms on $X$, which is a Cantor set, in particular it is compact and Hausdorff. Given a finite subtree of $F \subset T$, the complement $T \setminus F$ determines a partition $B_1 \sqcup \cdots \sqcup B_k$ of $X$. Let $\mathcal{I}$ denote the set of subsets of $X$ obtained as finite disjoint unions of such $B_i$'s. Then one can easily check that $\mathcal{I}$ satisfies ($\mathbf{C}$), that the action of $G$ satisfies ($2 \mathbf{T}$) and ($\mathbf{L}$). However, the action of $G$ on $T$ is of general type, so $G$ does not have Property $\ngt$. This does not contradict Theorem \ref{thm:cH:criterion}, because the action of $G$ on $X$ does not satisfy $(3 \mathbf{T})$.
\end{example}

\subsection{Proof of Theorem \ref{thm:cH:criterion}}

In order to prove the theorem, we will use the following algebraic criterion due to the third author \cite{anthony}. Our strategy is to show that the dynamical hypotheses from Theorem \ref{thm:cH:criterion} imply the algebraic conditions of this theorem. 

\begin{theorem}[{\cite[Theorem 1.1]{anthony}}]
\label{thm:criterion}

Let $G$ be a group. Suppose that there exist subsets $A, B \subset G$ and an integer $r \geq 1$ such that:
\begin{itemize}
    \item[(1)] $G$ is boundedly generated by $A$;
    \item[(2)] For every $b, b' \in B$ there exist $g, h \in G$ such that $b$ commutes with $b^g$, which in turn commutes with $b^h$, which in turn commutes with $b'$.
    \item[(3)] For every $g, h \in G$ there exist $b, b_1, \ldots, b_r \in B$ such that for every $a \in A$ there exists $f \in \langle b_1 \rangle \cdots \langle b_r \rangle$ such that each of $a^f b, a^f bg, a^f bh$ belongs to $B$.
\end{itemize}
Then $G$ has Property $\ngt$.
\end{theorem}

Let $G, X, \mathcal{I}$ be as in the statement of Theorem \ref{thm:cH:criterion}. If $X$ is finite, then so is $G$ as the action is faithful, in which case every action on a hyperbolic space has bounded orbits. So we may assume that $X$ is infinite. We will show that there exist subsets $A, B \subset G$ and $r \geq 1$ satisfying the three properties in Theorem \ref{thm:criterion}. We start with the following remark.

\begin{remark}
\label{rem:transitivity}

Double transitivity $(2 \mathbf{T})$ implies transitivity, namely $G$ acts transitively on $\mathcal{I}$. Indeed, if $I, I' \in \mathcal{I}$, then using that $\mathcal{I}$ is a basis and that $I, I'$ are not dense, we can choose $J, J' \in \mathcal{I}$ such that $(I, J), (I', J') \in \mathcal{I}^{(2)}$. Then $(2 \mathbf{T})$ gives an element $g \in G$ such that $(gI, gJ) = (I', J')$, in particular $gI = I'$.

Similarly, $(\mathbf{L})$ implies an analogous statement for two sets: Let $(I, J) \in \mathcal{I}^{(2)}$ and let $g, h \in G$ be such that $gI = I$ and $hJ = J$. Then there exists $b \in G$ such that $b|_I = g|_I$ and $b|_J = h|_J$. Once again, this follows by simply fixing $K$ to be any set in $\mathcal{I}$ such that $(I, J, K) \in \mathcal{I}^{(3)}$.
\end{remark}

We can immediately define $B$ and show that Property $(2)$ of Theorem \ref{thm:criterion} holds for this choice of~$B$. 

\begin{definition}[Choice of $B$]
We define $B$ to be the set of elements of $G$ that fix pointwise some $I \in \mathcal{I}$.
\end{definition}

\begin{lemma}[Property $(2)$]\label{lem:property2}
For every $b, b' \in B$ there exist $g, h \in G$ such that $b$ commutes with $b^g$, which in turn commutes with $b^h$, which in turn commutes with $b'$.
\end{lemma}

\begin{proof}
Let $I \in \mathcal{I}$ be fixed by $b$ and let $J \in \mathcal{I}$ be fixed by $b'$. By ($\mathbf{C}$), there exists $I' \in \mathcal{I}$ such that $I^c \subset I'$. Let $K \in \mathcal{I}$ be a subset of $I$, and let $g$ be an element such that $g I' = K$, which exists by $(2 \mathbf{T})$ (Remark \ref{rem:transitivity}). Using that $\mathcal{I}$ is a basis, we may choose $K$ in such a way that $K \cup J^c$ is not dense. Since $b$ is supported on $I'$, the conjugate $b^g$ is supported on $g I' \subset I$, which is fixed pointwise by $b$, so $b$ and $b^g$ commute.

Next, let $L \in \mathcal{I}$ be disjoint from both $K$ and $J^c$, which is possible since their union is not dense. Using $(2\mathbf{T})$ again, we find $h \in G$ such that $hI' = L$. Then by the same argument $b^h$ commutes with both $b^g$ and $b'$, since each of these pairs have disjoint support.
\end{proof}

For properties $(1)$ and $(3)$ we need to define the set $A$. For this, we first specify a particular subcollection $\mathcal{A}$ of the basis $\mathcal{I}$, as in the following lemma.

\begin{lemma}[Choice of $\mathcal{A}$]\label{choiceofma}
There exists a \emph{finite} subcover $\mathcal{A} \subset \mathcal{I}$ such that for every $I, J \in \mathcal{A}$ there exists $K \in \mathcal{A}$ such that $\overline{K}$ disjoint from $\overline{I} \cup \overline{J}$.
\end{lemma}

\begin{proof}
Since $X$ is infinite, there exist three distinct points $x, y, z \in X$. Since $X$ is compact and Hausdorff, it is normal, and since $\mathcal{I}$ is a basis there exist $U, V, W \in \mathcal{I}$ containing $x, y, z$ respectively, whose closures are pairwise disjoint. Now cover $X$ by elements $I \in \mathcal{I}$ whose closures intersect at most one of $\overline{U}, \overline{V}, \overline{W}$. By compactness, let $\mathcal{A}$ be a finite subcover of such elements, to which we add the three sets $U, V, W$. By construction $\mathcal{A} \subset \mathcal{I}$, and $\mathcal{A}$ is a finite cover of $X$. Finally, let $I, J \in \mathcal{A}$. Since each intersects at most one of $\overline{U}, \overline{V}, \overline{W}$, without loss of generality we may assume that they are both disjoint from $\overline{U}$. Then $U \in \mathcal{A}$ is such that $\overline{U}$ is disjoint from both $\overline{I}$ and $\overline{J}$.
\end{proof}

\begin{definition}[Choice of $A$]
We fix a choice of $\mathcal{A} \subset \mathcal{I}$ as in Lemma \ref{choiceofma}, and define $A$ to be the set of elements in $G$ that fix pointwise some $I \in \mathcal{A}$.
\end{definition}

Note that it follows from the definition that $A \subset B$. We now show that Property $(1)$ of Theorem \ref{thm:criterion} holds for this choice of $A$.

\begin{lemma}[Property $(1)$]
$G$ is boundedly generated by $A$. In fact, every element of $G$ may be written as a product of at most $3$ elements of $A$.
\end{lemma}

\begin{proof}
Let $g \in G$ and let $x, y \in X$ be such that $gx = y$. Suppose that $x \in I \in \mathcal{A}$ and $y \in J \in \mathcal{A}$ (where $I$ could be equal to $J$). Let $I' \in \mathcal{I}$ be small enough so that $x \in I' \subset I$ and $g I' = J' \subset J$; notice that $J' \in \mathcal{I}$ since $G$ preserves $\mathcal{I}$. Let $K \in \mathcal{A}$ be such that $\overline{K}$ is disjoint from $\overline{I} \cup \overline{J}$: this exists by the choice of $\mathcal{A}$ in Lemma \ref{choiceofma}. In particular neither $J \cup K$ nor $I' \cup K$ is dense, so we can apply ($2 \mathbf{T}$) to obtain $a_0 \in G$ such that $(a_0J, a_0K) = (I', K)$. By ($\mathbf{L}$) applied to $a_0$ and the identity acting on $J$ and $K$ respectively (see Remark \ref{rem:transitivity}), there exists an element $b \in G$ such that $b|_K = a_0|_K$ and $b|_J = id|_J$. Then the element $a := a_0b^{-1}$ satisfies $a|_K = id|_K$ and $aJ = I'$, still. In particular, this implies $a \in A$.

Now we have $ga J = J'$. Applying ($2 \mathbf{T}$) to the pair $(J', K), (J, K)$, there exists an element $c_0 \in G$ such that $(c_0J',c_0K) = (J,K)$. Using ($\mathbf{L}$) as before, we may modify $c_0$ to obtain an element $c$ such that $c J' = J$ and $c|_K = id|_K$, in particular $c \in A$. Now we have $cga J = J$, therefore using ($\mathbf{L}$) one more time, we obtain $d \in G$ such that $d|_K = id|_K$ and $d|_J = cga|_J$. In particular, $d \in A$. Thus $e := d^{-1}c \in A$, and moreover $ega|_J = id|_J$ so $ega \in A$. We conclude by observing that $g = e^{-1} . ega . a^{-1} \in A^3$.
\end{proof}

Lastly, we prove that Property $(3)$ of Theorem \ref{thm:criterion} holds for our choice of $A, B$, which will conclude the proof of Theorem \ref{thm:cH:criterion}.

\begin{lemma}[Property $(3)$]
There exists $r \geq 1$ with the following property: For every $g, h \in G$ there exist $b, b_1, \ldots, b_r \in B$ such that for every $a \in A$ there exists $f \in \langle b_1 \rangle \cdots \langle b_r \rangle$ such that each of $a^f b, a^f bg, a^f bh$ belongs to $B$.
\end{lemma}

\begin{proof}
Let $g,h \in G$. By ($3 \mathbf{T}$), there exist $(M, N, P) \in \mathcal{I}^{(3)}$ and $b_0 \in G$ such that $b_0 g M = M, b_0 h N = N$ and $b_0P =P$. Now we use ($\mathbf{L}$) twice to strengthen this. Applying ($\mathbf{L}$) to the actions of $b_0$ and $b_0 g$ on $P$ and $M$ respectively, gives an element $c_0 \in G$ such that $c_0|_P = b_0|_P, c_0|_M = b_0g|_M, c_0|_N = id|_N$. Again ($\mathbf{L}$) applied to the actions $b_0 h$ and the identity on $N$ and $M$ respectively, gives an element $c_1 \in G$ such that $c_1|_N = b_0 h|_N, c_1|_M = id|_M$, and $c_1|_P = id|_P$. Setting $b := c_1^{-1}c_0^{-1}b_0$, we obtain that $b|_P = id|_P, bg|_M = id|_M$ and $bh|_N = id|_N$.

Let $I \in \mathcal{A}$ be fixed. Let $Y \in \mathcal{I}$ be such that $\overline{Y}$ is disjoint from $M, N, P$: since their union is not dense in $X$ by definition of $\mathcal{I}^{(3)}$, such a $Y$ exists by the fact that $\mathcal{I}$ is a basis and that $X$ is a normal space. By ($\mathbf{C}$) there exists $J \in \mathcal{I}$ such that $Y^c \subset J$, which is equivalent to $J^c \subset Y$. By ($2\mathbf{T}$), there exists $f \in G$ such that $f I = J$ (see Remark \ref{rem:transitivity}), which yields $f(I^c) = (fI)^c = J^c \subset Y$. Notice that $f$ is chosen in terms of $I$ and $J$, $J$ is chosen in terms of $Y$, and that $Y$ is chosen in terms of $(M, N, P)$, which in turn are determined by $g$ and $h$. We denote the dependence as $f = f(g, h, I)$. Then for every element $a \in G$ fixing $I$ pointwise, the element $a^f$ fixes pointwise $M, N, P \subset Y^c \subset f I$. It follows that $a^fb$ is the identity on $P$, $a^f bg$ is the identity on $M$, and $a^f bh$ is the identity on $N$; in particular they are all in $B$.

Now the set of $f = f(g, h, I)$ that realize Property $(3)$ for the pair $(g, h)$ may be chosen to be finite, since $I$ varies in the finite set $\mathcal{A}$. Moreover, each element of $G$ may be written as a product of at most $3$ elements in $\mathcal{A}$, by Property $1$. It follows that there exist $3 |\mathcal{A}|$ elements in $A$ such that each $f = f(g, h, I)$ belongs to the product of the corresponding cyclic groups. Since $A \subset B$, and this bound is uniform, we obtain the desired conclusion.
\end{proof}

\subsection{Groups acting on Cantor sets}
\label{s:cantor}

The first application of the criterion from the last section is for group actions on Cantor sets. We start with the following definition. 

\begin{definition}
Let $G$ be a group acting on a Cantor set $X$. The \emph{topological full group of $G$}, denoted by $[[G]]$, is the group of all homeomorphisms $f$ with the following property: for every $x \in X$ there exists an open neighbourhood $U$ of $x$ and an element $g \in G$ such that $f|_U = g|_U$. Note that $G \leq [[G]]$.

We say that $G$ is a \emph{topological full group} if $G = [[G]]$. It follows from the definitions that for every $G$, the group $[[G]]$ is a topological full group, i.e. $[[\,[[ G ]]\,]] = [[G]]$.
\end{definition}

\begin{remark}
\label{rem:full}

The above definition can be equivalently given in terms of a basis. More precisely, suppose that $\mathcal{I}$ is a basis of $X$. Then $[[G]]$ is equal to the group of all homeomorphisms $f$ with the following property: for every $x \in X$, there exists $x \in U \in \mathcal{I}$ and an element $g \in G$ such that $f|_U = g|_U$.
\end{remark}

\begin{corollary}
\label{cor:cantor:criterion:firststep}
Let $G$ be a group acting faithfully on the Cantor set $X$. Suppose that there exists a basis $\mathcal{I}$ of proper non-empty clopen subsets of $X$ that is closed under taking complements, and is such that $G$ acts transitively on $\mathcal{I}^{(3)}$. Suppose moreover that $G$ is a topological full group. Then $G$ has Property $\ngt$.
\end{corollary}

\begin{proof}
We apply Theorem \ref{thm:cH:criterion} for the action of $G$ on $X$ with the basis $\mathcal{I}$. By assumption, $\mathcal{I}$ satisfies ($\mathbf{C}$), and the action of $G$ satisfies (2$\mathbf{T}$) and ($3 \mathbf{T}$). We are left to prove ($\mathbf{L}$). Let $(I, J, K) \in \mathcal{I}^{(3)}$ and let $g, h \in G$ be such that $g I = I$ and $h J = J$. Define $b : X \to X$ by $b|_I = g|_I, b|_J = h|_J$ and $b|_{X \setminus I \cup J} = id|_{X \setminus I \cup J}$. Since $G$ is a topological full group and $I, J$ are clopen, we obtain $b \in G$.
\end{proof}

\begin{corollary}
\label{cor:cantor:criterion}

Keep the assumptions from Corollary \ref{cor:cantor:criterion:firststep}. Suppose moreover that $G$ has finite abelianization, and that it acts transitively on $\mathcal{I}^{(4)}$. Then $G$ is hereditary $\nl$.
\end{corollary}

\begin{remark}
It will be apparent from the proof below that if $G$ is perfect, then the proof can be streamlined further. Also, in that case the assumption of transitivity on $\mathcal{I}^{(4)}$ will be not necessary; instead transitivity on $\mathcal{I}^{(3)}$ will be sufficient.
\end{remark}

\begin{proof}
We will show that the commutator subgroup $G'$ is simple and $\nl$, in particular it is hereditary $\nl$. This then exhibits $G$ as an extension of a hereditary $\nl$ group and a finite group, which is of course hereditary $\nl$ by Proposition \ref{prop:extension}.

We start by showing that $G'$ also satisfies the criterion of Theorem \ref{thm:cH:criterion}. Observe that $G'$ also acts faithfully on $X$. We will show that $G'$ acts transitively on $\mathcal{I}^{(3)}$, and that its action on $\mathcal{I}$ satisfies ($\mathbf{L}$), which is sufficient for our goal (note that we are not claiming that $G'$ is a topological full group). 

First, we prove that $G'$ acts transitively on $\mathcal{I}^{(3)}$. As a first step, we focus on the special case in which $(I, J, K), (I', J', K') \in \mathcal{I}^{(3)}$ are such that the union of the six sets is not dense in $X$. Since $\mathcal{I}$ is a basis and it is closed under complements, there exist $(M, N) \in \mathcal{I}^{(2)}$ such that $I, J, K \subset M$. Since $G$ is transitive on $\mathcal{I}^{(4)}$ by assumption, there exists $g \in G$ such that $g(I, J, K, N) = (I', J', K', N)$. Since $G$ is a topological full group, we may assume that $g$ fixes $N$ pointwise. Let $h \in G$ be such that $hN = M$. Then $h g^{-1} h^{-1}$ fixes $M$ pointwise, and therefore $[g, h] = g (h g^{-1} h^{-1})$ is an element of $G'$ that sends $(I, J, K)$ to $(I', J', K')$. Now let $(I, J, K), (I', J', K') \in \mathcal{I}^{(3)}$ be arbitrary. Then there exist $(I_1, J_1, K_1), (I_2, J_2, K_2) \in \mathcal{I}^{(3)}$ such that each of the pairs $\{ (I, J, K), (I_1, J_1, K_1) \}; \{ (I_1, J_1, K_1), (I_2, J_2, K_2) \}; \{ (I_2, J_2, K_2), (I', J', K') \}$ falls into the special case treated above: this follows from the fact that $\mathcal{I}$ is a basis, and that by definition each triple has non-dense union (see the proof of Lemma \ref{lem:property2}). Then we can apply the special case three times to obtain an element of $G'$ sending $(I, J, K)$ to $(I', J', K')$, which concludes the proof that the action of $G'$ on $\mathcal{I}^{(3)}$ is transitive. 

We are left to prove ($\mathbf{L}$). This follows from an argument similar to one given before. Namely let $(I, J, K) \in \mathcal{I}^{(3)}$ and let $g, h \in G'$ be such that $g I = I$ and $h J = J$. Since the union of $I, J, K$ is not dense, once again we find $(M, N) \in \mathcal{I}^{(2)}$ such that $I, J, K \subset M$. Define $b : X \to X$ by $b|_I = g|_I, b|_J = h|_J$ and $b|_{X \setminus I \cup J} = id|_{X \setminus I \cup J}$. Let $c \in G$ be such that $cN = M$. Then $[b, c] = b(cb^{-1}c^{-1})$ is an element of $G'$ such that $[c, b]|_I = g|_I, [c, b]|_J = h|_J$ and $[c, b]|_K = id|_K$.

It follows from Theorem \ref{thm:cH:criterion} that $G'$ has Property $\ngt$. We complete the proof by showing that $G'$ is uniformly simple, which shows that $G'$ is hereditary $\nl$ by Corollary \ref{cor:nl:uperfect}. For this, we can apply the criterion from \cite[Theorem 5.1]{usimple}, which states that if the action of $G$ on $X$ is \emph{extremely proximal}, then $G'$ is uniformly simple. By extremely proximal we mean: for every pair $U, V$ of nonempty proper clopen subsets, there exists $f \in G$ such that $f U \subset V$. So let $U, V$ be nonempty proper clopen subsets. Since $\mathcal{I}$ is a basis and it is closed under taking complements, there exist $I, J \in \mathcal{I}$ such that $U \subset I$ and $J \subset V$. By transitivity, there exists $g \in G$ such that $g I = J$, and thus $g U \subset g I = J \subset V$. This shows that $G$ is extremely proximal and concludes the proof.
\end{proof}

We now present a few notable examples of groups to which the above criteria apply. We remain a little superficial for now as details are provided in the next subsection on twisted Brin--Thompson groups; the arguments contained therein can be easily adapted for each of these examples as well.

\begin{example}[Higman--Thompson groups $V_n(r)$ \cite{higman, stein}]
\label{ex:V}

Let $n \geq 2, r \geq 1$, let $\mathcal{T}_n$ be a rooted $n$-ary tree, and $\mathcal{T}_n(r)$ the forest given as the disjoint union of $r$ copies of $\mathcal{T}_n$. Let $X$ be the boundary of $\mathcal{T}_n(r)$, so that $X$ is a Cantor set. For each finite subtree $F$ in one of the copies of $\mathcal{T}_n$, the complement $\mathcal{T}_n \setminus F$ induces a partition of $\partial \mathcal{T}_n \subset X$ into finitely many clopen sets $C_1, \ldots, C_k$, one for each connected component. We set $\mathcal{I}_0$ to be the set of $C_i$ that can be obtained this way, and we set $\mathcal{I}$ to be the set of proper disjoint unions of finitely many elements in $\mathcal{I}_0$.
Identifying $\mathcal{T}_n$ with $\{1, \ldots, n\}^\mathbb{N}$ induces a natural identification of each $C \in \mathcal{I}_0$ with $\partial \mathcal{T}_n$. Composing these identifications yields canonical homeomorphisms between any two elements of $\mathcal{I}_0$. The \emph{Higman--Thompson group} $V_n(r)$ is the group of homeomorphisms of $X$ obtained by fixing two finite partitions (of the same size) of $X$ into elements of $\mathcal{I}_0$ and permuting them according to those canonical homeomorphisms. When $r = 1$, we denote simply $V_n$.

It follows from the definitions that $V_n(r)$ is a topological full group, that $\mathcal{I}$ is preserved by $V_n(r)$, and that $\mathcal{I}$ is stable under taking complements (the latter property is why we use $\mathcal{I}$ and not $\mathcal{I}_0$). Moreover, given two elements in $\mathcal{I}^{(m)}$ it is easy to construct by hand an element of $V_n(r)$ sending one to the other.
Therefore the conditions of Corollary \ref{cor:cantor:criterion:firststep} are satisfied. Moreover, $V_n(r)$ has abelianization of order at most $2$ \cite[Theorem 5.4]{higman}, and therefore Corollary \ref{cor:cantor:criterion} applies.

As a special case, we recover \emph{Thompson's group} $V = V_2(1)$. However, we note that it was already shown in \cite{anthony} that $V$ is $\ngt$, and it is well-known that $V$ is uniformly simple (see e.g. \cite{usimple}).
\end{example}

\begin{example}[Coloured Neretin groups \cite{colouredneretin}]

Let $T$ be a $(d+1)$-regular (unrooted) tree, where $d \geq 2$, and let $F$ be a subgroup of the symmetric group of $D := \{ 0, \ldots, d \}$. We fix a proper colouring $c : E(T) \to D$, meaning that for all $v \in V(T)$, the restriction to the edges adjacent to $v$ is a bijection $c_v : E(v) \to D$. Given an element $g \in \mathrm{Aut}(T)$, and a vertex $v \in V(T)$, we denote by $\sigma(g, v)$ the \emph{local action of $g$ at $v$}, namely the permutation:
$$D \xrightarrow{c_v^{-1}} E(v) \xrightarrow{g} E(gv) \xrightarrow{c_{gv}} D.$$
The \emph{Burger--Mozes universal group} associated to $F$, denote by $U(F)$, is the subgroup of $\mathrm{Aut}(T)$ consisting of those elements $g$ such that $\sigma(g, v) \in F$ for all $v \in V(T)$ \cite{burgermozes}. The \emph{coloured Neretin group} associated to $F$, denoted by $\mathcal{N}_F$, is the topological full group of the action of $U(F)$ on $\partial T$. By definition it is a topological full group of homeomorphisms of the Cantor set $\partial T$, and moreover it has finite abelianization \cite[Theorem 1.2]{colouredneretin}. Now $\partial T$ is naturally isomorphic to $\partial \mathcal{T}_d(d+1)$, with the notation of the previous example, and the action of $\mathcal{N}_F$ preserves the same basis $\mathcal{I}$. Finally, if $F$ is transitive, then $\mathcal{N}_F$ contains a copy of $V_{d, d+1}$ whose action on $\partial T$ is conjugate to its action on $\partial \mathcal{T}_d(d+1)$ described in the previous example \cite[Remark 3.18]{colouredneretin} (see also \cite{colouredneretin0, colouredneretin1}). Therefore the transitivity properties are also satisfied, and Corollary \ref{cor:cantor:criterion} applies: $\mathcal{N}_F$ is hereditary $\nl$.
\end{example}

\begin{example}[R\"{o}ver--Nekrashevych groups \cite{MR1714140, MR2119267}]
\label{ex:RN}

Let $\mathcal{T}_n$ be again a rooted $n$-ary tree, and let $G \leq \mathrm{Aut}(\mathcal{T}_n)$ be \emph{self-similar} (i.e., for every $g \in G$, if $(g_1, \ldots, g_n, \sigma)$ denotes the image of $g$ under the canonical isomorphism $\mathrm{Aut}(\mathcal{T}_n) \to \mathrm{Aut}(\mathcal{T}_n) \wr_{\{1, \ldots,n \}} S_n$, then $g_1, \ldots, g_n$ also belong to $G$). Then we can modify the definition of $V_n$ as given above, by including not only canonical homeomorphisms between elements of $\mathcal{I}_0$, but also canonical homeomorphisms twisted by the action of $G$ on $\mathcal{T}_n$. Namely, while the canonical homeomorphisms are given by the canonical identifications $C_1 \to \partial \mathcal{T}_n \to C_2$, homeomorphisms twisted by $G$ are given by
$$C_1 \to \partial \mathcal{T}_n \xrightarrow{g \in G} \partial \mathcal{T}_n \to C_2.$$
This defines the \emph{R\"{o}ver--Nekrashevych group} $V_n(G)$. Such groups interpolate between Thompson groups $V_n$ and \emph{Neretin's groups} $\mathrm{AAut}(\mathcal{T}_n):=V_n(\mathrm{Aut}(\mathcal{T}_n))$; see \cite{neretin} or \cite{waldi} for an introduction to the topic.

Since $V_n$ is a subgroup of $V_n(G)$ and both preserve $\mathcal{I}$, the transitivity properties are carried over from $V_n$ to $V_n(G)$. Moreover it again follows from the definitions that $V_n(G)$ is a topological full group, so Corollary \ref{cor:cantor:criterion:firststep} applies and $V_n(G)$ is $\ngt$ for all $n$ and all $G$.

Finally, many R\"{o}ver--Nekrashevych groups have finite abelianisations, or even are virtually simple \cite{MR2119267, MR3822291, MR3910073}. Concrete examples include Neretin's groups \cite{kapoudjian}, Thompson's groups $V_n$ (as above), and the groups $V_n(\mathfrak{G})$ constructed from Grigorchuk groups $\mathfrak{G}$ \cite{MR1714140}. It follows from Corollary \ref{cor:cantor:criterion} that all these examples are hereditary $\nl$.
\end{example}

\begin{example}[Symmetrizations $QV_n(r)$ \cite{QThompson}]\label{ex:SymmetrizationV}
Let $n \geq 2, r \geq 1$ and let $\mathcal{T}_n(r)$ denote the forest given by the disjoint union of $r$ copies of rooted $n$-ary trees. We endow each vertex with a fixed total order on its children. Let $QV_n(r)$ denote the group of bijections of the vertex set $\mathcal{T}_n(r)$ preserving adjacency and orders with only finitely many exceptions. The group of finitely supported permutations of the vertices of $\mathcal{T}_n(r)$, which we denote by $S_\infty$, yields a normal subgroup in $QV_n(r)$. The corresponding quotient is isomorphic to Thompson's group $V_n(r)$. Thus, $QV_n(r)$ fits in the short exact sequence
$$1 \to S_\infty \to QV_n(r) \to V_n(r) \to 1.$$
By combining the Example \ref{ex:V} (and the fact that torsion groups are hereditary $\nl$) with Proposition~\ref{prop:extension}, it follows that $QV_n(r)$ is hereditary $\nl$. 
\end{example}

\begin{remark}[Other Thompson--like groups]

Many other variations of $V$ are present in the literature, in particular the \emph{Stein--Thompson groups} $V_{n_1, \ldots, n_k}$, or more generally the \emph{Stein--Higman--Thompson groups} $V_{n_1, \ldots, n_k}(r)$ \cite{stein}; and the \emph{golden ratio Thompson group} $V_\tau$. However, these groups are defined in terms of tree pair diagrams that do not fit into a natural infinite rooted tree, and so they have no natural interpretation as groups of homeomorphisms of Cantor sets. This is discussed in \cite[Section 1]{ttau} for $V_\tau$, and similar discussions apply to the other groups.

Nevertheless, we believe that it should be possible to verify the algebraic criterion from Theorem \ref{thm:criterion} for such groups, in terms of tree pair diagrams with rich enough dynamics. Together with the fact that they have finite abelianization \cite{stein, ttau}, this would lead to a proof that they are hereditary $\nl$. However one main goal for this paper was to highlight the \emph{dynamical} approach to Property $\nl$, that works beyond Cantor sets as we will later see. A \emph{combinatorial} criterion in terms of tree pair diagrams would be of great interest, but it falls out of the scope of this paper.
\end{remark}

We end with one last example that serves as the basic building block for twisted Brin--Thompson groups.

\begin{example}[Brin--Thompson groups $sV$ \cite{brin}]

Let $s \geq 1$ be a natural number. Let $C \subset [0, 1]$ be the middle-third Cantor set, which can be identified with $\{0, 1\}^{\mathbb{N}}$ via ternary expansion. Then $X := C^s$ is also a Cantor set, seen as a subset of $[0, 1]^s$ with the product topology. A subdivision of $X$ into clopen subsets obtained by consecutively dividing $X$ in half along one of the $s$ coordinates is called a \emph{pattern}, and the corresponding clopen subsets are called \emph{bricks}. Elements of $sV$ are homeomorphisms of $X$ obtained by fixing patterns with the same number of bricks on domain and codomain, and sending each brick of the domain pattern to a brick of the codomain pattern affinely and preserving the orientation. For $s = 1$, we get the classical Thompson group $V$. We refer the reader to \cite{brin} for more details.

Again, it follows easily from the definitions that $sV$ is a topological full group. For each pattern with bricks $B_1, \ldots, B_n$, we declare a proper disjoint union of finitely many $B_i$ to be in $\mathcal{I}$, and let $\mathcal{I}$ be the set of all clopen subsets of $X$ obtained this way. Then $\mathcal{I}$ is closed under taking complements, and again transitivity on $\mathcal{I}^{(4)}$ may be easily checked. Finally, it is known that $sV$ is simple \cite[Theorem 1]{brin}, in particular it is perfect. Therefore it follows from Corollary \ref{cor:cantor:criterion} that $sV$ is hereditary $\nl$.
\end{example}

\subsection{Twisted Brin--Thompson groups}

In this section, we aim to prove the following result, which is related to Corollary \ref{cor:qi}. 

\begin{theorem}
\label{thm:qi}

Every finitely generated group quasi-isometrically embeds into a finitely generated simple $\nl$ group.
\end{theorem}

This will be a direct consequence of properties of the twisted Brin--Thompson groups, and the following application of Corollary \ref{cor:cantor:criterion}, which we prove below. 

\begin{proposition}

\label{cor:twistedV}

Let $\Gamma$ be a group, $S$ a countable set on which $\Gamma$ acts faithfully. Then the twisted Brin--Thompson group $SV_\Gamma$ is hereditary $\nl$.
\end{proposition}

We start with the definitions, while referring the reader to \cite{twistedV, twistedV2} for more details. Let $C := \{ 0, 1 \}^{\mathbb{N}}$ be the standard binary Cantor set, and let $X := C^S$, which is a Cantor set when equipped with the product topology. We denote its elements as functions $\kappa : S \to C$. Let $\{ 0, 1 \}^*$ be the set of all \emph{finite} binary sequences, including the empty word $\varnothing$. The \emph{support} of a map $\psi : S \to \{0, 1\}^*$ is the set of $s \in S$ such that $\psi(s) \neq \varnothing$.
Given a function $\psi : S \to \{0, 1\}^*$ with finite support, the corresponding \emph{(dyadic) brick} is defined as
$$B(\psi) :=  \{ \kappa \in C^S : \psi(s) \text{ is a prefix of } \kappa(s) \text{ for all } s \in S \}.$$
Since $\psi$ has finite support, this imposes clopen conditions on only finitely many coordinates, and therefore each $B(\psi)$ is a clopen subset of $X$. Every brick is canonically homeomorphic to $C^S$, via the map $h_\psi : C^S \to B(\psi)$ defined by $h_\psi(\kappa)(s) = \psi(s)\kappa(s)$ (concatenation). This defines, for each pair of functions with finite support $\psi, \varphi$, a canonical homeomorphism $h_\varphi \circ h_\psi^{-1} : B(\psi) \to B(\varphi)$.

A \emph{pattern} is a partition of $X$ into finitely many bricks. We define $SV$ to be the group of homeomorphisms of $X$ obtained by fixing patterns with the same numbers of bricks on domain and codomain, and sending each brick of the domain pattern to a brick of the codomain pattern via a canonical homeomorphism $h_\varphi \circ h_\psi^{-1}$. This is a direct generalization of the Brin--Thompson group $sV$ to an infinite-dimensional setting.

Now let $\Gamma$ be a group acting faithfully on $S$. For each $\gamma \in \Gamma$, let $\tau_\gamma$ be the homeomorphism of $X$ defined by $\tau_\gamma(\kappa)(s) = \kappa(\gamma^{-1} . s)$. This defines, for each pair of functions of finite support $\psi, \varphi$ and each $\gamma \in \Gamma$, a \emph{twist homeomorphism} $h_\varphi \circ \tau_\gamma \circ h_\psi^{-1} : B(\psi) \to B(\varphi)$. The \emph{twisted Brin--Thompson group} $SV_\Gamma$ is defined like $SV$ above, but bricks are sent to each other via twist homeomorphisms instead of only canonical homeomorphisms.

\begin{remark}
In \cite{twistedV} the group is defined in terms of dyadic partitions, which are more restrictive than patterns. However, the above description yields the same group \cite[Remark 1.1]{twistedV}
\end{remark}

\begin{example}
When $\Gamma = \{ 1 \}$, we have $SV_\Gamma = SV$. Therefore Brin--Thompson groups $sV$ and the classical Thompson group $V$ are special cases of twisted Brin--Thompson groups.
\end{example}

\begin{proof} [Proof of Theorem \ref{thm:qi}] Assuming we have proven Proposition \ref{cor:twistedV}, Theorem \ref{thm:qi} follows by using the following facts:
\begin{itemize}
    \item[1.] $SV_\Gamma$ is simple \cite[Theorem 3.4]{twistedV};
    \item[2.] If $\Gamma$ is finitely generated and the action of $\Gamma$ on $S$ has finitely many orbits, then $SV_\Gamma$ is finitely generated \cite[Theorem A]{twistedV};
    \item[3.] In the above case, the embedding $\Gamma \to SV_\Gamma : \gamma \to \tau_\gamma$ is quasi-isometric \cite[Theorem B]{twistedV}.
\end{itemize}

Therefore in order to obtain a quasi-isometric embedding of an arbitrary finitely generated group $\Gamma$ into a twisted Brin--Thompson group, one may take $S = \Gamma$ acting on itself by left translation. \end{proof}

It therefore remains to prove Proposition \ref{cor:twistedV}. 
\begin{proof}[Proof of Proposition \ref{cor:twistedV}]
We will show that the action of $G := SV_\Gamma$ on $X = C^S$ satisfies the hypotheses of Corollary \ref{cor:cantor:criterion}. We have already mentioned that $SV_\Gamma$ is simple, and in particular it is perfect (in fact, the proof \cite[Theorem 3.4]{twistedV} works by verifying conditions similar to those that we are going to check now).

Let
$$\mathcal{I} := \{ B = B_1 \sqcup \cdots \sqcup B_k : B_i \text{ is a dyadic brick, } \varnothing \neq B \neq X \}.$$
Then $\mathcal{I}$ is a basis, and it is closed under taking complements, since every set of disjoint bricks can be completed into a set of disjoint bricks forming a partition of $X$. Every $g \in G$ is defined in terms of a pattern $B_1 \sqcup \cdots \sqcup B_k$ being sent to another pattern, where each brick is sent to a brick. Any other brick $B$ can be partitioned in terms of its intersection with each $B_i$, which shows that $B$ is sent to a finite disjoint union of bricks. Therefore $G$ preserves $\mathcal{I}$. Since $\mathcal{I}$ is a basis, the same argument implies that $G$ is a topological full group (see Remark \ref{rem:full}).

Finally, consider sets of disjoint bricks $B_1, \ldots, B_k$ and $B_1', \ldots, B_k'$ whose union is not all of $X$. Each of them can be completed into a partition of $X$ of the form $B_1, \ldots, B_l$ and $B_1', \ldots, B_l'$ with $l > k$: the fact that the cardinality is the same may be achieved by taking refinements. Then there exists an elment of $SV \leq SV_\Gamma = G$ sending $B_i$ to $B_i'$. This shows that the action is highly transitive on proper disjoint unions of bricks, and it follows that the action on $\mathcal{I}^{(n)}$ is transitive for all $n \geq 1$. This concludes the verification of the hypotheses of Corollary \ref{cor:cantor:criterion}, and thus the proof.
\end{proof}

\subsection{Groups acting on the circle}

In this subsection, we apply Theorem \ref{thm:cH:criterion} to groups acting on the circle. Here we also need to take into account the orientation, which is why the statement of Theorem \ref{thm:cH:criterion} only requires a weak version of triple transitivity. The proof is essentially the same as in the case of Cantor sets, but some extra care has to be taken at various steps, in particular the proof of property ($3 \mathbf{T}$) from Theorem \ref{thm:cH:criterion}, and in the proof of virtual uniform simplicity.

Throughout this section, we will denote the circle by $X$, which will be endowed with a fixed orientation. Given $a \neq b \in X$, we denote by $(a, b)$ the open arc oriented from $a$ to $b$, that is the set of all $x \in X$ such that $(a, x, b)$ is positively oriented. Similarly, we define $[a, b] = \overline{(a, b)}$ the closed arc oriented from $a$ to $b$.

\begin{definition}
Let $G$ be a group acting faithfully on a circle $X$ preserving the orientation. Let $\mathcal{O} \subset X$ be a subset preserved by $G$. We say that $G$ is an \emph{$\mathcal{O}$-piecewise full group} if it contains all homeomorphisms $f$ with the following property: there exists a positively oriented tuple $(o_1, \ldots, o_n) \in \mathcal{O}^n$, such that for every $i = 1, \ldots, n$ there exists $g_i \in G$ such that $f|_{[o_i, o_{i+1}]} = g_i|_{[o_i, o_{i+1}]}$ (where of course we read $n+1 = 1$). 
\end{definition}

\begin{remark}
The above definition is not standard. Topological full groups can be defined for actions on the circle as well, but for our purposes we also need to be able to glue together local actions on closed arcs that agree at their endpoints.
\end{remark}

\begin{corollary}
\label{cor:circle:criterion:firststep}

Let $G$ be a group acting faithfully on the circle $X$ preserving the orientation. Suppose that there exists a dense orbit $\mathcal{O} \subset X$ such that $G$ acts transitively on positively oriented $n$-tuples in $\mathcal{O}$, for $n \leq 6$. Suppose moreover that $G$ is an $\mathcal{O}$-piecewise full group. Then $G$ has Property $\ngt$.
\end{corollary}

\begin{proof}
Let $\mathcal{I}$ be the set open arcs in $X$ with (distinct) endpoints in $\mathcal{O}$; in symbols $\mathcal{I} = \{ (a, b) : a, b, \in \mathcal{O}, a \neq b \}$. $\mathcal{I}$ is a basis because $\mathcal{O}$ is dense an all of its elements are neither dense nor empty. Further it clearly satisfies ($\mathbf{C}$). As $G$ preserves $\mathcal{O}$, so it also preserves $\mathcal{I}$. Let $(I, J) \in \mathcal{I}^{(2)}$, say $I = (a, b)$ and $J = (c, d)$. Since their closures are disjoint, the tuples $(a, b, c, d)$ and $(c, d, a, b)$ are positively oriented, so they are in the same $G$-orbit by the assumption on high transitivity. Choosing $g \in G$ such that $g(a, b, c, d) = (c, d, a, b)$ gives $g I = J$ and proves ($2 \mathbf{T}$).

For ($3 \mathbf{T}$), let $g, h \in G$. Without loss of generality we may assume that $g \neq h$. Let $x \in \mathcal{O}$ be such that $gx \neq hx$: this exists because $\mathcal{O}$ is dense. Then $(gx, hx)$ is an open arc between two distinct points, thus it is not the whole circle $X$: let $\delta$ be the diameter of the complement, for some fixed metric on $X$. Then, for $y$ close enough to $x$ in the positive direction, the arc $(x, y)$ has diameter less than $\delta/3$, and the arc $(gx, hy)$ still has complement of diameter more than $2 \delta / 3$. It follows that there exists $z \in X$ such that $(x, y, z)$ and $(gx, hy, z)$ are both positively oriented, by choosing $z$ in the complement of $[x, y] \cup [gx, hy] \neq X$. Now we let $M, N, P$ be small neighbourhoods of $x, y, z$ in $\mathcal{I}$ that are pairwise disjoint, have non-dense union, and such that $g M, h N$ and $P$ are also pairwise disjoint. The same proof as in the previous paragraph, using transitivity on sextuples, allows to find $b \in G$ such that $bg M = M, bh N = N$ and $b P = P$, proving ($3 \mathbf{T}$).

We are left to prove ($\mathbf{L}$). Let $(I, J, K) \in \mathcal{I}^{(3)}$ and let $g, h \in G$ be such that $g I = I$ and $h J = J$. Let $a,b,c,d \in \mathcal{O}$ be such that $I = (a,b)$ and $J = (c,d)$. The tuple $(a,b,c,d)$ is positively oriented, and since $G$ acts preserving the orientation, both endpoints of $I$ are fixed by $g$, and both endpoints of $J$ are fixed by $h$. We define $t : X \to X$ by $t|_I = g|_I, t|_J = h|_J$ and $t|_{X \setminus I \cup J} = id|_{X \setminus I \cup J}$, from which it follows that $t|_K = id|_K$. This is well-defined and a homeomorphism, since $g$ and $h$ fix the endpoints of $I$ and $J$. Since $G$ is an $\mathcal{O}$-piecewise full group, $t \in G$. This concludes the proof.
\end{proof}

As in the case of Cantor sets, our next corollary will add some hypotheses to ensure that the group $G$ is virtually uniformly simple, and thus hereditary $\nl$. For the sake of completeness, we prove a general criterion based on \emph{rigid stabilizers}: which are the subgroups $G(x)$ for some point $x \in X$ consisting of elements in $G$ that fix pointwise some open neighbourhood of $x$. However, let us point out that in some of the specific cases that we will consider, uniform simplicity results are already available in the literature (see e.g. \cite{liousse, spectrum}).

\begin{corollary}
\label{cor:circle:criterion}
Keep the assumptions from Corollary \ref{cor:circle:criterion:firststep}. Suppose moreover that there exists $x \in \mathcal{O}$ such that $G(x)$ has finite abelianization, and that $G$ acts transitively on positively ordered $n$-tuples in $\mathcal{O}$, for $n \leq 8$. Then $G$ is hereditary $\nl$
\end{corollary}

\begin{remark}
As in the case of Corollary \ref{cor:cantor:criterion}, it will be apparent from the proof below that if $G(x)$ is perfect, then the proof can be streamlined. Also, in that case the assumption on transitivity on $8$-tuples is not needed as transitivity on $6$-tuples is sufficient.
\end{remark}

\begin{proof}
Let $H := \langle G(x)' : x \in \mathcal{O} \rangle \leq G$. Notice that $g G(x)' g^{-1} = G(g . x)'$, therefore $H$ is normal in $G$. We will show that $H$ is simple and $\nl$ and has finite-index in $G$. Once again, this exhibits $G$ as an extension of hereditary $\nl$ groups, and so Proposition \ref{prop:extension} will give us the required conclusion.

We start by showing that $H$ has $\ngt$ by applying Theorem \ref{thm:cH:criterion}. We cannot directly apply Corollary \ref{cor:circle:criterion:firststep}, since $H$ need not be an $\mathcal{O}$-piecewise full group. So to prove that it has $\ngt$ we will show that $H$ acts transitively on positively ordered $6$-tuples in $\mathcal{O}$, and that it satisfies ($\mathbf{L}$). This is enough, since ($\mathbf{C}$) is given, and ($2 \mathbf{T}$) and ($3 \mathbf{T}$) just follow from $6$-transitivity (see the proof of Corollary \ref{cor:circle:criterion:firststep}).

As a first step to show transitivity on  $6$-tuples, let $(x_1, \ldots, x_6), (y_1, \ldots, y_6)$ be positively ordered $6$-tuples such that there exists $z \in \mathcal{O}$ such that $(z, x_1, \ldots, x_6)$ and $(z, y_1, \ldots, y_6)$ are still positively ordered. Choose elements $z_-, z_+, w_-, w_+ \in \mathcal{O}$ such that the $11$-tuple
$$(z_-, z, w_-, w_+, z_+, x_1, \ldots, x_6)$$
is positively oriented, and the same holds for the $y_i$. By $8$-transitivity, there exists $g \in G$ such that $g(z_-, z_+, x_1, \ldots, x_6) = (z_-, z_+, y_1, \ldots, y_6)$, and since $G$ is an $\mathcal{O}$-piecewise full group, we may assume that $g$ is the identity on $[z_-, z_+]$, in particular $g \in G(z)$. Now by $4$-transitivity there exists $h \in G$ such that $h(z_-, w_-, w_+, z_+) = (z_-, w_-, x_1, x_6)$. Again, we may assume that $h$ is the identity on $(z_-, w_-)$ and so $h \in G(z)$. Then $h g^{-1} h^{-1}$ is the identity on $[x_1, x_6]$, in particular it fixes all of the $x_i$, and so $[g, h] = g (hg^{-1}h^{-1})$ is an element of $G(z)'$ that sends $(x_1, \ldots, x_6)$ to $(y_1, \ldots, y_6)$. For the general case, it suffices to notice that any two triples may be sent to each other by a sequence of instances of the above special case, just as in the proof of Lemma \ref{lem:property2}. The proof of ($\mathbf{L}$) is similar to the one provided in the presvious proof (see also the proof of Corollary \ref{cor:cantor:criterion}).

It follows that $H$ has Property $\ngt$. We now show that $H$ has finite index in $G$. First, we claim that if $g \in G$ and $x, y \in \mathcal{O}$ are distinct points such that $g(y) \notin \{ x, y \}$, then $g \in G(x) G(y)$. Indeed, let $I \in \mathcal{I}$ be such that $y \in I, x \notin \overline{I}$ and $x, y \notin \overline{g I}$. Since $G(x)$ is transitive on ordered pairs in $\mathcal{O} \setminus \{ x \}$ by the previous paragraph, we find an element $f \in G(x)$ such that $f I = g I$. Let $h$ be an element supported on $I$ such that $h|_I = f^{-1} g|_I$, which exists because $G$ is an $\mathcal{O}$-piecewise full group: this way $h \in G(x)$, since $x \notin \overline{I}$. We then have $h^{-1} f^{-1} g|_I = id|_I$, in particular $h^{-1} f^{-1} g \in G(y)$, and since $f h \in G(x)$ we obtain $g = (f h) (h^{-1} f^{-1} g) \in G(x) G(y)$.

Now, fix distinct points $x, y$ and let $s \in G$ be such that $sx, sy \notin \{ x, y \}$. If $g \in G$ is such that $gy \in \{x, y \}$, then $(sg)y \notin \{x, y\}$. It follows from the above claim that $G = G(x)G(y) \cup s^{-1} G(x) G(y)$. Moreover, by hypothesis there exist finite sets $F_x \subset G(x), F_y \subset G(y)$ such that $G(x) = F_x G(x)'$ and $G(y) = F_y G(y)'$. Thus:
\begin{align*}
    G &= G(x) G(y) \cup s^{-1} G(x) G(y) \\
    &= \{1, s^{-1}\} . \left( \bigcup\limits_{f_x \in F_x} f_x G(x)' \right) . \left( \bigcup\limits_{f_y \in F_y} f_y G(y)' \right) \\
    &= \{1, s^{-1}\} . \bigcup\limits_{f_x \in F_x, f_y \in F_y} f_x G(x)' f_y G(y)' \\
    &= \{1, s^{-1}\} . \bigcup\limits_{f_x \in F_x, f_y \in F_y} f_x f_y G(f_y^{-1} . x)' G(y)' \\
    &\subset \{ 1, s^{-1} \} . F_x . F_y . H.
\end{align*}
It follows that $H$ has finite index in $G$. More precisely, the above equations show that there exists a finite set $F$ and $x_f, y_f \in \mathcal{O}$ for each $f \in F$ such that $G = \bigcup_{f \in F} f G(x_f)' G(y_f)'$. But if $g = f G(x_f)' G(y_f)' \in H$, then it follows that $f \in H$ too. Therefore $H = \bigcup_{f \in F \cap H} f G(x_f)' G(y_f)'$. Since $F \cap H$ is finite, from the definition of $H$ we obtain that $H$ is boundedly generated by the subgroups $\{ G(x)' : x \in \mathcal{O} \}$.

We end by showing that $H$ is uniformly simple, which implies that $H$ is hereditary $\nl$ by Corollary \ref{cor:nl:uperfect}, and allows to conclude the proof. This will follow the same outline as the proof of uniform simplicity of $T_{\tau}'$ from \cite{spectrum:preprint}, and implies the conclusion by an argument similar to the one given in the proof of Corollary \ref{cor:cantor:criterion}.

We can see $G(x)$ as a subgroup of the homeomorphism group of the open interval $Y := X \setminus \{ x \}$, which is moreover \emph{boundedly supported}, in the sense that every element has support contained in a compact subset of $Y$. Moreover, the action of $G(x)$ on $Y$ is transitive on ordered quadruples in $\mathcal{O} \cap Y$: this follows from the fact that $G$ is transitive on positively ordered sextuples in $X$, and it is an $\mathcal{O}$-piecewise full group. We can now apply \cite[Theorem 3.1]{usimple}, which states that if a group with a boundedly supported action on a linearly ordered set is \emph{proximal}, then its commutator subgroup is $6$-uniformly simple. By proximal we mean: for all $a < b$ and all $c < d$ there exists a group element $g$ such that $ga < c < d < gb$. Proximality is easily implied by high transitivity on a dense subset, similarly to the proof of Corollary \ref{cor:cantor:criterion}, so we conclude that $G(x)'$ is $6$-uniformly simple for every $x \in \mathcal{O}$.

Finally, to show uniformly simple, let $g \in H$. We will show that for every $x \in \mathcal{O}$ and every $f \in G(x)'$, we can write $f$ as a product of at most $24$ conjugates of $g, g^{-1}$. Since $H$ is boundedly generated by subgroups of the form $G(x)'$, it will then follow that $H$ is uniformly simple, which concludes the proof. For this, let $I \in \mathcal{I}$ be such that $I \cap g I = \varnothing$, and $x \notin \overline{I \cup g I}$. Let $h \in G$ be a non-trivial element supported on $I$. Since $I \cap g I = \varnothing$, it follows that $k = [g, h] = h^g . h^{-1}$ is non-trivial, and supported on $I \cup g I$, in particular $k \in G(x)$. Notice that $k = g . (g^{-1})^h$ is a product of two conjugates of $g, g^{-1}$. Choosing an element $s \in G(x)$ that does not commute with $k$ we have that $l := [s, k] \in G(x)'$ and $l$ can be written as a product of $4$ conjugates of $g, g^{-1}$ as follows: $l = [s, k] = k^s k^{-1} = g^s . (g^{-1})^{hs} . g^h . g^{-1}$.
By $6$-uniform simplicity of $G(x)'$, we can write $f$ as a product of $6$ conjugates of $l, l^{-1}$. Therefore $f$ is a product of $24$ conjugates of $g, g^{-1}$, which concludes the proof.
\end{proof}

\begin{example}[Stein--Thompson groups \cite{stein}]
\label{ex:T}

Let $n_1, n_2, \ldots, n_k \geq 1$ and let $\lambda := \prod n_i$. The \emph{Stein--Thompson group} $T_{n_1, \ldots, n_k}$ is the group of orientation-preserving homeomorphisms of the circle that preserve $\mathcal{O} := \mathbb{Z}[\lambda]/\mathbb{Z} \subset \mathbb{R}/\mathbb{Z} = X$, are piecewise linear functions with breakpoints in $\mathcal{O}$, and slopes in $P := \langle n_1, \ldots, n_k \rangle$.

Suppose that $n_1 = 2$ and $n_2 = 3$. Then $T_{2, 3, \ldots, n_k}$ acts highly transitively on $\mathcal{O}$ and is an $\mathcal{O}$-piecewise full group: the high transitivity can be proven as in \cite[Example 3.7]{monsters}, and relies on the fact that one of the slope generators is $2$. Moreover, $T_{2, 3, n_3, \ldots, n_k}(x)$ is perfect for all $x \in \mathcal{O}$: this follows from the perfection criterion in \cite[Theorem 2.14]{PL}, as is verified in \cite[Lemma 7.2]{liousse}. Therefore Corollary \ref{cor:circle:criterion} implies that $T_{2, 3, n_3, \ldots, n_k}$ is hereditary $\nl$.

In a simpler way, the above arguments holds for Thompson's group $T = T_2$, where all of the necessary results are well-known and can also be found in \cite{monsters, PL, liousse}. Therefore $T$ is hereditary $\nl$, which gives a strong positive answer to \cite[Question 1.5]{anthony}.
\end{example}

Note that for the above examples, uniform simplicity was already proven in \cite{liousse}.

\begin{example}[The symmetrization $QT$]
We saw in Example~\ref{ex:SymmetrizationV} how to define symmetrizations $QV = QV_2(1)$. A similar construction is possible for $QT$. The shortest definition is to take the pre-image of $T \leq V$ in $QV$ under the projection $QV \twoheadrightarrow V$. Equivalently, this amounts to fixing an embedding of the rooted binary tree $\mathcal{T}$ into $\mathbb{R}^2$ and to considering the group of bijections $\mathcal{T}^{(0)} \to \mathcal{T}^{(0)}$ induced by isotopies of the plane and preserving adjacency and left-right orders on children with only finitely many exceptions. Because of the short exact sequence
$$1 \to S_\infty \to QT \to T \to 1,$$
the combination of the previous example with Proposition~\ref{prop:extension} shows that $QT$ is hereditary $\nl$. 
\end{example}

\begin{remark}[A note on other Thompson-like groups]
    For slopes other than $2$, our arguments do not apply because of the lack of high transitivity. For instance, consider the group $F_3$ acting on the interval $(0, 1)$. The action preserves $\mathbb{Z}[1/3] \cap (0, 1)$, but there are two orbits: $\{ a/3^k : a \text{ even} \}$ and $\{ a/3^k : a \text{ odd}\}$. The corresponding $T$-like group is $T_3$, which acts on the circle preserving $\mathbb{Z}[1/3]/\mathbb{Z}$. This time the presence of rotations implies that the action is transitive, however it is not \emph{doubly transitive}, because the action of the stabilizer of $0$ on $(\mathbb{R}/\mathbb{Z} \setminus \{ 0 \})$ is conjugate to the action of $F_3$ on the open interval described above, which has two orbits in $\mathbb{Z}[1/3] \cap (0, 1)$

    Similar arguments apply for other slopes. For the same reason we cannot adapt our results to groups such as $QT_n$. Therefore in this case one needs a more careful argument in order to prove property $\nl$, and even $\ngt$.
\end{remark}

\begin{example}[Golden ratio Thompson group $T_\tau$ \cite{ttau}]

Let $\tau$ be the small golden ratio: $\tau = \frac{\sqrt{5}-1}{2}$. Let $\mathcal{O} := \mathbb{Z}[\tau]/\mathbb{Z} \subset \mathbb{R}/\mathbb{Z}$, and let $T_\tau$ be the group of orientation-preserving homeomorphisms of the circle that preserve $\mathcal{O}$, are piecewise linear with breakpoints in $\mathcal{O}$ and slopes in $\tau^{\mathbb{Z}}$. By results from \cite{ftau, ftau1, spectrum}, $T_{\tau}$ acts highly transitively on $\mathcal{O}$, is an $\mathcal{O}$-piecewise full group, and has abelianization of order $2$. Therefore Corollary \ref{cor:circle:criterion} implies that $T_{\tau}$ is hereditary $\nl$.
\end{example}

We end with an example that goes beyond the piecewise linear setting, to showcase the flexibility of our criterion. Since the definitions and structural properties of the groups involved go beyond the scope of this paper, we limit ourselves to giving precise references for each statement.

\begin{example}
Let $S$ be the finitely presented simple group of piecewise projective homeomorphisms of the circle constructed by Lodha in \cite{yashS}. By definition, Thompson's group $T$ is a subgroup of $S$, seen in its piecewise projective realization preserving the set $\mathcal{O}$ of rational points on the projective line, and $S$ preserves $\mathcal{O}$ as well. Therefore $S$ acts highly transitively on $\mathcal{O}$, since $T$ does.

Given $x \in \mathcal{O}$, the group $S(x)$ is a subgroup of the finitely presented piecewise projective group $G_0$ constructed by Lodha and Moore in \cite{yashG0}. In fact, it follows by comparing the standard forms for $S$ described in \cite[Section 3.2]{yashS} to the standard forms for $G_0$ described in \cite[Section 2.5]{yashS}, that $S(x)$ coincides with the subgroup of $G_0$ consisting of elements with compact support, such that the total sum of exponents of the $y$-generators equals $0$. By \cite[Proposition 2.2]{yashG0'} this is precisely the group $G_0'$, which is simple; in particular it is perfect \cite[Theorem 2(1)]{yashG0'}. Thus Corollary \ref{cor:circle:criterion} applies, and $S$ is hereditary $\nl$.
\end{example}

\section{Connections with other properties}\label{sec:otherprops}

In this last section, we summarize the relations between Property $\nl$ and other properties of groups. Some of these connections are fairly easy to deduce from the contents of this paper, but we record them here for possible future work. 

\subsection{Fixed point properties} Since simplicial (resp. real) trees are hyperbolic spaces, Property $\nl$ is related to Property $(\mathrm{FA})$ (resp. $(\mathrm{F}\mathbb{R})$), i.e.\ every action of the group under consideration on a simplicial (resp.\ real) tree has a global fixed point. This is recorded in the next lemma.

\begin{lemma} 
Suppose $G$ has Property $\nl$. Then $G$ does not split non-trivially as an HNN extension or an amalgamated free product. Moreover, if $G$ is not the directed union of countably many proper subgroups (in particular, if $G$ is finitely generated) then it satisfies the properties $(\mathrm{FA})$ and $(\mathrm{F}\mathbb{R})$.
\end{lemma}

\begin{proof} The first assertion follows from the fact that (non-trivial) amalgamated free products and HNN extensions act on simplicial trees with loxodromics; see for instance \cite{Serre}. Next, let $G$ act on a real tree $T$. Because $G$ has $\nl$, the action must be either elliptic or horocyclic. In the former case, $G$ has a global fixed point. In the latter case, $G$ globally fixes a point at infinity $\xi \in \partial T$ and all its elements are elliptic (as there are no parabolic isometries for an action on a tree). As a consequence, every element of $G$ fixes pointwise some infinite ray pointing to $\xi$. Fixing an infinite ray $\rho$ pointing to $\xi$, it follows that $G$ can be written as the union of the fixators $\mathrm{Fix}(\rho_n)$, where $\rho_n$ denotes the infinite ray obtained from $\rho$ by removing an initial segment of length $n$. Thus, if $G$ is not the directed union of countably many proper subgroups, then there exists some $n \geq 0$ such that $\mathrm{Fix}(\rho_n) = G$, which therefore has a global fixed point. 
\end{proof}

A fixed point property can also be deduced from $\nl$ for higher dimensional analgues of trees, namely finite dimensional CAT(0) cube complexes. As an immediate consequence of \cite[Theorem~5.1]{anthony}, we have the following.

\begin{proposition}
If a group is hereditary $\nl$, then every action on a finite-dimensional CAT(0) cube complex has a global fixed point.
\end{proposition}

As shown in \cite{anthony}, this proposition is a consequence of the fact that finite-dimensional CAT(0) cube complexes are built from hyperbolic spaces. In the same spirit, one can reasonably expect that every action of an $\nl$ group on a hierarchically hyperbolic space has no loxodromic. This assertion does not follow immediately from the existing literature, and a proof would go beyond the scope of this article, so we only record the following observation. 

\begin{proposition}
A hierarchically hyperbolic group is hereditary $\nl$ if and only if it is finite.
\end{proposition}

\begin{proof}
Finite groups are obviously hereditary $\nl$. So let $(G,\mathfrak{S})$ be a hierarchically hyperbolic group, and suppose that $\mathfrak{S}$ contain an unbounded domain. According to \cite[Theorem~3.2]{HHG}, there exists a finite and $G$-invariant collection of pairwise orthogonal domains $\mathcal{E}(G):= \{W_1, \ldots, W_n\} \subset \mathfrak{S}$ such that every unbounded domain in $\mathfrak{S}$ is nested in some element of $\mathcal{E}(G)$. Let $H \leq G$ be a finite-index subgroup. We may assume without loss of generality that $H$ stabilizes each domain in $\mathcal{E}(G)$. So $H$ acts on each hyperbolic space $\mathcal{C}W_i$. Because the projection $G \to \mathcal{C}W_i$ is coarsely surjective by definition, this action must be cobounded. Since $H$ is $\nl$ and cobounded actions are never horocyclic, it follows that each $\mathcal{C}W_i$ is bounded, which is a contradiction. Thus $\mathfrak{G}$ does not contain an unbounded domain. We conclude (for instance from the distance formula) that $G$ is bounded, which amounts to saying that $G$ is finite.
\end{proof}

\subsection{Actions on products} Given an integer $k \geq 1$, say that a group $G$ satisfies Property $(\mathrm{NL}_k)$ if every isometric action on a product of $\leq k$ hyperbolic graphs has no loxodromic element. Here, we equip the product space $\prod X_i$ with the $\ell^1$ metric and by a loxodromic element, we mean an element $g \in G$ such that the orbit map $n \to g^n.x$ is a quasi-isometric embedding for some (equivalently any) basepoint $x \in \prod X_i$.

\begin{proposition}
A group $G$ has $(\mathrm{NL}_k)$ if and only if every subgroup of index $\leq k$ in $G$ has $\nl$. 
\end{proposition}

\begin{proof}
Assume that $G$ does not have $(\mathrm{NL}_k)$, i.e.\ $G$ admits an action on a product $X= X_1 \times \cdots \times X_r$ of $r \leq k$ hyperbolic graphs with some loxodromic element $g \in G$. According to \cite[Theorem~6.1]{Button}, $G$ preserves the product structure of $X$, possibly permuting the isomorphic factors. Consequently, the stabilizer $H$ of a factor, which has index $\leq r$ in $G$, acts on a hyperbolic graph with a loxodromic element (namely a power of $g$ belonging to $H$). So $G$ contains a subgroup of index $\leq k$ which does not satisfy $\nl$.

\medskip 
Conversely, assume that $G$ contains a subgroup $H$ of index $r \leq k$ admitting an action on a hyperbolic space $X$ with a loxodromic element $h \in H$. Up to replacing $X$ by the graph whose vertices are the points in $X$ and whose edges connect any two points at distance $\leq 1$, we can assume without loss of generality that $X$ is a graph. Then, the action of $H$ on $X$ classically extends to an action of $G$ on $X^r$, and $h$ remains loxodromic. So $G$ fails to have $(\mathrm{NL}_k)$.
\end{proof}

\begin{corollary}
A group is hereditary $\nl$ if and only if every isometric action on a product of finitely many hyperbolic graphs has no loxodromic element.
\end{corollary}

\subsection{Hyperbolic structures} In \cite{ABO}, the authors study the poset of hyperbolic structures on a group $G$, denoted $\HG$. The poset consists of equivalence classes of \emph{cobounded} isometric actions on hyperbolic spaces, with a partial order that roughly corresponds to collapsing equivariant families of subspaces to get the smaller action from the larger. Equivalently, one can also think of hyperbolic structures as equivalence classes of (not necessarily finite) generating sets of $G$. As the precise definition of the poset is not necessary for this paper, we refer the reader to \cite[Section 3]{ABO} 
 for definitions and state only the following relevant result here.  

\begin{theorem}[{\cite[Theorem 4.6]{ABO}}]
For any group $G$, $$\HG = \Hl_e(G) \sqcup \Hl_\ell(G) \sqcup \Hl_{qp}(G) \sqcup \Hl_{gt}(G),$$ where $\Hl_e(G), \Hl_\ell(G), \Hl_{qp}(G), \Hl_{gt}(G)$ denote the subposets of elliptic, lineal, focal and general type actions. Moreover, $\Hl_e(G) = \{ [G]\}$ for any group $G$, and this is referred to as the trivial structure. 
\end{theorem} 

As the discussion of hyperbolic structures involves considering only cobounded actions, we could define cobounded versions of the properties considered in this paper. For instance, Property $\ngt_{cb}$ is the property that no \emph{cobounded} action of a group on a hyperbolic space is of general type, and $\nne_{cb}$ and $\nl_{cb}$ defined analogously. However, it turns out that these cobounded versions of the properties are equivalent to the properties themselves -- this will be proved in the appendix (Corollary \ref{cor:coboundedgt}), since the result and proof are of independent interest. The result also has strong implications for the relation between the properties studied in this paper and the structure of $\HG$, summerized by the following corollary.

\begin{corollary}\label{cbdabtHG}
\begin{enumerate}
    \item G has Property $\ngt$ if and only if $\mathcal{H}_{gt}(G) = \emptyset$;
    \item G has Property $\nne$ if and only if $\HG = \Hl_e(G) \sqcup \Hl_\ell(G)$;
    \item G has Property $\nl$ if and only if $\HG$ is trivial. 
    \end{enumerate}
\end{corollary}

\begin{proof}
    For each implication, one direction is immediate; the other implication is a direct application of Corollary \ref{cor:coboundedgt}.
\end{proof}

It is also worth noting that the structure of the poset $\HG$ is not necessarily preserved by finite-index subgroups. 

\begin{example}\label{ex:finindexposet}
The group $\mathbb{D}_\infty \times \mathbb{D}_\infty$ has 3 hyperbolic structures. However, it contains a index 4 subgroup isomorphic to $\Z \times \Z$, which has uncountably many hyperbolic structures; see \cite[Theorem 2.3 and Example 4.23]{ABO}.  \end{example}

An open question from \cite{ABO} consequently seeks to explore under what conditions the structure of $\HG$ might be preserved by finite-index subgroups. A partial answer is provided by the following easy observation. 

\begin{corollary} Let $G$ be a hereditary $\nl$ group. Then the structure of $\HG$ is preserved by finite-index subgroups. \end{corollary}

\appendix
\section{Appendix: Passing to a cobounded action (by Alessandro Sisto)}

The goal of this appendix is to describe a construction of cobounded actions on hyperbolic spaces starting from possibly non-cobounded ones. This construction preserves many properties of the original action, such as being of general type. The key result is the following, whose proof roughly speaking says that, starting with an action on a hyperbolic space and a quasi-convex subset of an orbit, we can cone-off all geodesics far away from the orbit to obtain a new hyperbolic space that still contains a copy of the quasiconvex set, and furthemore geodesics are ``preserved'' by this coning-off.

\begin{proposition}\label{prop:cobounded}
 Let the group $G$ act on the hyperbolic space $X$, let $o\in X$, and let $Q\geq 0$. Then there exists $K>0$, some hyperbolic space $Y$ with a cobounded $G$-action and a $G$-equivariant coarsely Lipschitz map $\pi:X\to Y$ such that $\pi|_Z$ is a $(K,K)$-quasi-isometric embedding for any $Q$-quasiconvex subspace $Z$ of $G\cdot o$. Moreover, there exists $D$ such that given any geodesic $[x,y]$ in $X$, we have that $\pi([x,y])$ lies at Hausdorff distance at most $D$ from any geodesic with endpoints $\pi(x)$, $\pi(y)$.
\end{proposition}

\begin{proof}
 Up to replacing $X$ with the graph whose vertex-set is $X$ and whose edges connect any two points at distance $\leq 1$, we can and shall assume that $X$ is a graph for convenience.

For any given $R>0$, let $Y_R$ denote the graph obtained from $X$ by adding an edge between any two vertices lying on some geodesic disjoint from $N_R(G\cdot o)$. Note that there is a natural 1-Lipschitz map $\pi_R:X\to Y_R$. In view of \cite[Corollary 2.4]{KR} (which roughly speaking says that coning-off quasiconvex subspaces in a hyperbolic space yields another hyperbolic space with the ``same'' geodesics) we have that $Y_R$ is hyperbolic and that the ``moreover'' part holds for any $R$. In fact, the same corollary yields that the ``moreover'' part holds with constant $D$ independent of $R$. Note also that $G$ acts on $Y_R$, that $\pi_R$ is $G$-equivariant, and that the action on $Y_R$ is cobounded.

Therefore, fixing $Q>0$, what is left to show is that given a $Q$-quasiconvex subspace $Z$ of $G\cdot o \subset X$, for any sufficiently large $R$ we have that $\pi_R$ restricts to an isometric embedding on $Z$ (recall that we modified $X$ to be a graph at the beginning of the proof, and this is why the conclusion of the theorem only gives a quasi-isometric embedding). We can in fact take any $R>Q+D$, since geodesics in $Y=Y_R$ with endpoints on $\pi_R(Q)$ cannot then cross any edge of $Y_R$ which is not an edge of $X$. Indeed, it is readily seen that the $\pi_R$-image of the $R$-neighborhood of $G\cdot o$ is the $R$-neighborhood of $G\cdot \pi_R(o)$ in $Y$, and the ``moreover'' part and quasiconvexity imply that any geodesic in $Y$ connecting points of $\pi_R(Z)$ is contained in the $R$-neighborhood of $G\cdot \pi_R(o)$.
\end{proof}

We now point out that the proposition actually allows us to conclude that the action on $Y$ satisfy analogous properties to those of the action of $X$.

\begin{corollary}
\label{cor:preserve}
 Let the group $G$ act on the hyperbolic space $X$, and let $g_1,\dots,g_n\in G$. Consider the following properties:
 \begin{enumerate}[noitemsep]
  \item the elements $g_i$ are loxodromic,
  \item the elements $g_i$ are loxodromic WPD,
  \item the elements $g_i$ are independent loxodromic elements.
 \end{enumerate}

 There exists some hyperbolic space $Y$ with a cobounded $G$-action, and such that each property (1)-(3) holds in $Y$ if it holds in $X$.
\end{corollary}

We do not know whether the construction in Proposition \ref{prop:cobounded} preserves acylindricity, meaning that if $G \acts X$ is acylindrical, we do not known if $G \acts Y_R$ is also acylindrical, for $Y_R$ as in the proof of the proposition. However, it follows from \cite[Theorem 1.2]{Osinah} that if $G$ has an acylindrical, general type action on a hyperbolic space, then it also admits a cobounded, acylindrical, general type action on a (different) hyperbolic space (or even a quasi-tree; see \cite{Balqt,BBFS}). Related to this, we note that the fact that a group admits an action with loxodromic WPD elements on a hyperbolic space if and only if it admits one such action which is furthermore cobounded (a direct consequence of Corollary \ref{cor:preserve}(2)) is an important part of the proof of \cite[Theorem 1.2]{Osinah}, but our argument is more elementary.

\begin{proof}[Proof of Corollary \ref{cor:preserve}]
Item (1) follows from the proposition taking $Q$ large enough that the orbits of all $\langle g_i\rangle$ are $Q$-quasiconvex. For item (3), the union of the orbits of all $\langle g_i\rangle$ is $Q$-quasiconvex for some $Q$, and the fact that this quasi-isometrically embeds in $Y$ implies that the limit points of the $g_i$ in $Y$ are distinct.

Item (2) is slightly more complicated (and not needed later). For ease of notation, let us consider a single loxodromic WPD element $g$, and consider a space $Y$ from the proposition where $g$ still acts loxodromically, with the associated constant $D$, and suppose that the equivariant map $\pi$ is $D$-coarsely Lipschitz. It is a consequence of e.g. \cite[Corollary 4.4]{S:contr} that there exists a virtually cyclic subgroup $E(g)$ containing $g$ and with the property that there exists $C>0$ such that for any $h\notin E(g)$ any two geodesics $\gamma_1,\gamma_2$ from $h\langle g\rangle o$ to $\langle g\rangle o$ contain points $p_i\in\gamma_i$ with $d_X(p_1,p_2)\leq C$.

Let us now verify the WPD property for $g$ acting on $Y$. Fix any $R>0$, and consider any integer $n>0$ such that $d_Y(o',g^no')> DC+3D+2R$, where $o'=\pi(o)$. We want to show that there are only finitely many $h\in G$ such that $d_Y(o',ho'),d_Y(g^no',hg^no')\leq R$. It suffices to show that any such $h$ must belong to $E(g)$, as within each right $\langle g\rangle$-coset in $E(g)$ there are only finitely many elements $h$ satisfying $d_Y(o',ho')\leq R$. 

Suppose by contradiction that $h\notin E(g)$, and consider geodesics $\gamma_1$ from $o$ to $ho$ and $\gamma_2$ from $g^no$ to $hg^no$ (both geodesics in $X$). On one hand the ``moreover'' part of Proposition \ref{prop:cobounded} implies that $d_Y(\pi(\gamma_1),\pi(\gamma_2))> DC+D$. On the other hand, the discussion above and the fact that $\pi$ is $D$-coarsely Lipschitz yield points $p_i\in\gamma_i$ with $d_Y(\pi(p_1),\pi(p_2))\leq DC+D$, a contradiction. Hence, $h\in E(g)$, as required.
\end{proof}

\begin{corollary}
\label{cor:coboundedgt}
Let $G$ be a group acting on some hyperbolic space $X$, and assume that the action is not horocyclic. Then there exists a hyperbolic space $Y$ on which $G$ acts coboundedly such that $G \curvearrowright X$ and $G\curvearrowright Y$ have the same type.
\end{corollary}

\begin{proof}
If $G \curvearrowright X$ is elliptic, lineal, or focal, then its orbits are quasiconvex and there exists a Cayley graph of $G$ quasi-isometric to the orbit -- this follows from a slightly altered version of the Svarc-Milnor Lemma applied to quasi-geodesics of a fixed quality. This gives us the required cobounded action. If the action is of general type, then this follows from item (3) of the previous corollary.
\end{proof}

\footnotesize

\addcontentsline{toc}{section}{References}
\bibliographystyle{alpha}
\bibliography{references}

\begin{thebibliography}{FFLZ22b}

\bibitem[ABO19]{ABO}
C.~R. Abbott, S.~H. Balasubramanya, and D.~Osin.
\newblock Hyperbolic structures on groups.
\newblock {\em Algebr. Geom. Topol.}, 19(4):1747--1835, 2019.

\bibitem[ABR21]{ABR}
C.~R. Abbott, S.~H. Balasubramanya, and A.~J. Rasmussen.
\newblock Higher rank confining subsets and hyperbolic actions of solvable
  groups.
\newblock {\em Adv. Math.}, To appear. \textit{arXiv:2108.08175}, 2021.

\bibitem[AR19]{AR}
C.~R. Abbott and A.~J. Rasmussen.
\newblock Actions of solvable {B}aumslag-{S}olitar groups on hyperbolic metric
  spaces.
\newblock {\em Algebr. Geom. Topol.}, To appear. \textit{arXiv:1906.04227},
  2019.

\bibitem[Bal17]{Balqt}
S.~H. Balasubramanya.
\newblock Acylindrical group actions on quasi-trees.
\newblock {\em Algebraic \& Geometric Topology}, 17:2145--2176, 2017.

\bibitem[Bal20]{BalLamp}
S.~H. Balasubramanya.
\newblock Hyperbolic structures on wreath products.
\newblock {\em J. Group Theory}, 23(2):357--383, 2020.

\bibitem[BBFS19]{BBFS}
M.~Bestvina, K.~Bromberg, K.~Fujiwara, and A.~Sisto.
\newblock Acylindrical actions on projection complexes.
\newblock {\em Enseign. Math.}, 65 no.1-2:1--32, 2019.

\bibitem[BCFS22]{lattices}
U.~Bader, P.-E. Caprace, A.~Furman, and A.~Sisto.
\newblock Hyperbolic actions of higher-rank lattices come from rank-one
  factors.
\newblock {\em arXiv preprint arXiv:2206.06431}, 2022.

\bibitem[Bes22]{bestvinasurvey}
M.~Bestvina.
\newblock Groups acting on hyperbolic spaces--a survey.
\newblock {\em arXiv preprint arXiv:2206.12916}, 2022.

\bibitem[BH99]{bridson_haefliger}
M.~R. Bridson and A.~Haefliger.
\newblock {\em Metric spaces of non-positive curvature}, volume 319 of {\em
  Grundlehren der mathematischen Wissenschaften [Fundamental Principles of
  Mathematical Sciences]}.
\newblock Springer-Verlag, Berlin, 1999.

\bibitem[BLR18]{yashG0'}
J.~Burillo, Y.~Lodha, and L.~Reeves.
\newblock Commutators in groups of piecewise projective homeomorphisms.
\newblock {\em Adv. Math.}, 332:34--56, 2018.

\bibitem[BM99]{bm1}
M.~Burger and N.~Monod.
\newblock Bounded cohomology of lattices in higher rank {L}ie groups.
\newblock {\em J. Eur. Math. Soc. (JEMS)}, 1(2):199--235, 1999.

\bibitem[BM00]{burgermozes}
M.~Burger and S.~Mozes.
\newblock Groups acting on trees: from local to global structure.
\newblock {\em Inst. Hautes \'{E}tudes Sci. Publ. Math.}, (92):113--150 (2001),
  2000.

\bibitem[BM02]{bm2}
M.~Burger and N.~Monod.
\newblock Continuous bounded cohomology and applications to rigidity theory.
\newblock {\em Geom. Funct. Anal.}, 12(2):219--280, 2002.

\bibitem[BNR21]{ftau1}
J.~Burillo, B.~Nucinkis, and L.~Reeves.
\newblock An irrational-slope {T}hompson's group.
\newblock {\em Publ. Mat.}, 65(2):809--839, 2021.

\bibitem[BNR22]{ttau}
J.~Burillo, B.~Nucinkis, and L.~Reeves.
\newblock Irrational-slope versions of {T}hompson's groups {$T$} and {$V$}.
\newblock {\em Proc. Edinb. Math. Soc. (2)}, 65(1):244--262, 2022.

\bibitem[Bow06]{MR2243589}
B.~H. Bowditch.
\newblock {\em A course on geometric group theory}, volume~16 of {\em MSJ
  Memoirs}.
\newblock Mathematical Society of Japan, Tokyo, 2006.

\bibitem[Bra00]{MR1776048}
N.~Brady.
\newblock Finite subgroups of hyperbolic groups.
\newblock {\em Internat. J. Algebra Comput.}, 10(4):399--405, 2000.

\bibitem[Bre07]{MR2402591}
E.~Breuillard.
\newblock On uniform exponential growth for solvable groups.
\newblock {\em Pure Appl. Math. Q.}, 3(4, Special Issue: In honor of Grigory
  Margulis. Part 1):949--967, 2007.

\bibitem[Bri04]{brin}
M.~G. Brin.
\newblock Higher dimensional {T}hompson groups.
\newblock {\em Geom. Dedicata}, 108:163--192, 2004.

\bibitem[BS16]{PL}
R.~Bieri and R.~Strebel.
\newblock {\em On groups of {PL}-homeomorphisms of the real line}, volume 215
  of {\em Mathematical Surveys and Monographs}.
\newblock American Mathematical Society, Providence, RI, 2016.

\bibitem[But22]{Button}
J.~O. Button.
\newblock Generalised {B}aumslag-{S}olitar groups and hierarchically hyperbolic
  groups.
\newblock {\em arXiv preprint arXiv:2208.12688}, 2022.

\bibitem[BZ22]{twistedV}
J.~Belk and M.~C.~B. Zaremsky.
\newblock Twisted {B}rin-{T}hompson groups.
\newblock {\em Geom. Topol.}, 26(3):1189--1223, 2022.

\bibitem[Cal09]{scl}
D.~Calegari.
\newblock {\em scl}.
\newblock Number~20. Mathematical Society of Japan, 2009.

\bibitem[CCMT15]{Amen}
P.-E. Caprace, Y.~Cornulier, N.~Monod, and R.~Tessera.
\newblock Amenable hyperbolic groups.
\newblock {\em J. Eur. Math. Soc. (JEMS)}, 17(11):2903--2947, 2015.

\bibitem[CDM11]{colouredneretin0}
P.-E. Caprace and T.~De~Medts.
\newblock Simple locally compact groups acting on trees and their germs of
  automorphisms.
\newblock {\em Transform. Groups}, 16(2):375--411, 2011.

\bibitem[CDP90]{MR1075994}
M.~Coornaert, T.~Delzant, and A.~Papadopoulos.
\newblock {\em G\'{e}om\'{e}trie et th\'{e}orie des groupes}, volume 1441 of
  {\em Lecture Notes in Mathematics}.
\newblock Springer-Verlag, Berlin, 1990.
\newblock Les groupes hyperboliques de Gromov. [Gromov hyperbolic groups], With
  an English summary.

\bibitem[CKW94]{MR1306556}
D.~I. Cartwright, V.~A. Ka\u{\i}manovich, and W.~Woess.
\newblock Random walks on the affine group of local fields and of homogeneous
  trees.
\newblock {\em Ann. Inst. Fourier (Grenoble)}, 44(4):1243--1288, 1994.

\bibitem[Cle00]{ftau}
S.~Cleary.
\newblock Regular subdivision in {$\mathbf Z[\frac{1+\sqrt 5}{2}]$}.
\newblock {\em Illinois J. Math.}, 44(3):453--464, 2000.

\bibitem[FFL21a]{spectrum:preprint}
F.~Fournier-Facio and Y.~Lodha.
\newblock Algebraic irrational stable commutator length in finitely presented
  groups.
\newblock {\em arXiv preprint arXiv:2110.06286}, 2021.

\bibitem[FFL21b]{spectrum}
F.~Fournier-Facio and Y.~Lodha.
\newblock Second bounded cohomology of groups acting on $1 $-manifolds and
  applications to spectrum problems.
\newblock {\em Adv. Math.}, To appear. \textit{arXiv:2111.07931}, 2021.

\bibitem[FFLZ22a]{bV}
F.~Fournier-Facio, Y.~Lodha, and M.~C.~B. Zaremsky.
\newblock Braided {T}hompson groups with and without quasimorphisms.
\newblock {\em Algebr. Geom. Topol.}, To appear. \textit{arXiv:2204.05272},
  2022.

\bibitem[FFLZ22b]{monsters}
F.~Fournier-Facio, Y.~Lodha, and M.~C.~B. Zaremsky.
\newblock Finitely presented left orderable monsters.
\newblock {\em Ergodic Theory Dynam. Systems}, To appear.
  \textit{arXiv:2211.05268}, 2022.

\bibitem[Fri17]{frigerio}
R.~Frigerio.
\newblock {\em Bounded cohomology of discrete groups}, volume 227 of {\em
  Mathematical Surveys and Monographs}.
\newblock American Mathematical Society, Providence, RI, 2017.

\bibitem[Gen19]{anthony}
A.~Genevois.
\newblock Hyperbolic and cubical rigidities of {T}hompson's group {$V$}.
\newblock {\em J. Group Theory}, 22(2):313--345, 2019.

\bibitem[GG17]{usimple}
S.~R. Gal and J.~Gismatullin.
\newblock Uniform simplicity of groups with proximal action.
\newblock {\em Trans. Amer. Math. Soc. Ser. B}, 4:110--130, 2017.
\newblock With an appendix by N. Lazarovich.

\bibitem[GL15]{neretin}
L.~Garncarek and N.~Lazarovich.
\newblock The {N}eretin groups.
\newblock {\em arXiv preprint arXiv:1502.00991}, 2015.

\bibitem[GLA21]{liousse}
N.~Guelman, I.~Liousse, and P.~Arnoux.
\newblock Uniform simplicity for subgroups of piecewise continuous bijections
  of the unit interval.
\newblock {\em Bull. Lond. Math. Soc.}, To appear. \textit{arXiv:2109.05706},
  2021.

\bibitem[GM08]{GrovesManning}
D.~Groves and J.~F. Manning.
\newblock Dehn filling in relatively hyperbolic groups.
\newblock {\em Israel J. Math.}, 168:317--429, 2008.

\bibitem[Gro78]{MR474887}
J.~R.~J. Groves.
\newblock Soluble groups with every proper quotient polycyclic.
\newblock {\em Illinois J. Math.}, 22(1):90--95, 1978.

\bibitem[Gro87]{Gro}
M.~Gromov.
\newblock Hyperbolic groups.
\newblock In {\em Essays in group theory}, volume~8 of {\em Math. Sci. Res.
  Inst. Publ.}, pages 75--263. Springer, New York, 1987.

\bibitem[Gru57]{gruenberg}
K.~W. Gruenberg.
\newblock Residual properties of infinite soluble groups.
\newblock {\em Proc. London Math. Soc. (3)}, 7:29--62, 1957.

\bibitem[GST17]{GromovMonster}
D.~Gruber, A.~Sisto, and R.~Tessera.
\newblock Gromov's random monsters do not act non-elementarily on hyperbolic
  spaces.
\newblock {\em Proc. Amer. Math. Soc.}, To appear. \textit{arXiv:1705.10258},
  2017.

\bibitem[Hae20]{haettel}
T.~Haettel.
\newblock Hyperbolic rigidity of higher rank lattices.
\newblock {\em Ann. Sci. \'{E}c. Norm. Sup\'{e}r. (4)}, 53(2):439--468, 2020.
\newblock With an appendix by V. Guirardel and C. Horbez.

\bibitem[Hig74]{higman}
G.~Higman.
\newblock {\em Finitely presented infinite simple groups}.
\newblock Notes on Pure Mathematics, No. 8. Australian National University,
  Department of Pure Mathematics, Department of Mathematics, I.A.S., Canberra,
  1974.

\bibitem[Kap99]{kapoudjian}
C.~Kapoudjian.
\newblock Simplicity of {N}eretin's group of spheromorphisms.
\newblock {\em Ann. Inst. Fourier (Grenoble)}, 49(4):1225--1240, 1999.

\bibitem[KR14]{KR}
I.~Kapovich and K.~Rafi.
\newblock On hyperbolicity of free splitting and free factor complexes.
\newblock {\em Groups Geom. Dyn.}, 8(2):391--414, 2014.

\bibitem[LB17]{colouredneretin1}
A.~Le~Boudec.
\newblock Compact presentability of tree almost automorphism groups.
\newblock {\em Ann. Inst. Fourier (Grenoble)}, 67(1):329--365, 2017.

\bibitem[Led17]{waldi}
W.~Lederle.
\newblock {\em Almost Automorphism groups of trees}.
\newblock PhD thesis, ETH Z{\"u}rich, 2017.

\bibitem[Led19]{colouredneretin}
W.~Lederle.
\newblock Coloured {N}eretin groups.
\newblock {\em Groups Geom. Dyn.}, 13(2):467--510, 2019.

\bibitem[Leh08]{QThompson}
J.~Lehnert.
\newblock {\em Gruppen von quasi-Automorphismen}.
\newblock PhD thesis, Goethe Universit{\"a}t Frankfurt am Main, 2008.

\bibitem[LM16]{yashG0}
Y.~Lodha and J.~T. Moore.
\newblock A nonamenable finitely presented group of piecewise projective
  homeomorphisms.
\newblock {\em Groups Geom. Dyn.}, 10(1):177--200, 2016.

\bibitem[LMR95]{nofreesubsemigroups}
P.~Longobardi, M.~Maj, and A.~H. Rhemtulla.
\newblock Groups with no free subsemigroups.
\newblock {\em Trans. Amer. Math. Soc.}, 347(4):1419--1427, 1995.

\bibitem[Lod19]{yashS}
Y.~Lodha.
\newblock A finitely presented infinite simple group of homeomorphisms of the
  circle.
\newblock {\em J. Lond. Math. Soc. (2)}, 100(3):1034--1064, 2019.

\bibitem[Man08]{Man}
J.~F. Manning.
\newblock Actions of certain arithmetic groups on {G}romov hyperbolic spaces.
\newblock {\em Algebr. Geom. Topol.}, 8(3):1371--1402, 2008.

\bibitem[Mar16]{martelli}
B.~Martelli.
\newblock An introduction to geometric topology.
\newblock {\em arXiv preprint arXiv:1610.02592}, 2016.

\bibitem[Mon22]{lamplighters}
N.~Monod.
\newblock Lamplighters and the bounded cohomology of {T}hompson's group.
\newblock {\em Geom. Funct. Anal.}, 32(3):662--675, 2022.

\bibitem[Nek04]{MR2119267}
V.~Nekrashevych.
\newblock Cuntz-{P}imsner algebras of group actions.
\newblock {\em J. Operator Theory}, 52(2):223--249, 2004.

\bibitem[Nek18]{MR3822291}
V.~Nekrashevych.
\newblock Finitely presented groups associated with expanding maps.
\newblock In {\em Geometric and cohomological group theory}, volume 444 of {\em
  London Math. Soc. Lecture Note Ser.}, pages 115--171. Cambridge Univ. Press,
  Cambridge, 2018.

\bibitem[Osi16]{Osinah}
D.~Osin.
\newblock Acylindrically hyperbolic groups.
\newblock {\em Trans. Amer. Math. Soc.}, 368(2):851--888, 2016.

\bibitem[PS23]{HHG}
H.~Petyt and D.~Spriano.
\newblock Unbounded domains in hierarchically hyperbolic groups.
\newblock {\em Groups Geom. Dyn.}, 17(2):479--500, 2023.

\bibitem[Ros74]{supramenable}
J.~M. Rosenblatt.
\newblock Invariant measures and growth conditions.
\newblock {\em Trans. Amer. Math. Soc.}, 193:33--53, 1974.

\bibitem[R{\"o}v99]{MR1714140}
C.~R{\"o}ver.
\newblock Constructing finitely presented simple groups that contain
  {G}rigorchuk groups.
\newblock {\em J. Algebra}, 220(1):284--313, 1999.

\bibitem[Ser77]{Serre}
J.-P. Serre.
\newblock {\em Arbres, amalgames, {${\rm SL}_{2}$}}.
\newblock Soci\'{e}t\'{e} Math\'{e}matique de France, Paris, 1977.
\newblock Avec un sommaire anglais, R\'{e}dig\'{e} avec la collaboration de
  Hyman Bass, Ast\'{e}risque, No. 46.

\bibitem[Sis18]{S:contr}
A.~Sisto.
\newblock Contracting elements and random walks.
\newblock {\em J. Reine Angew. Math.}, 742:79--114, 2018.

\bibitem[Ste92]{stein}
M.~Stein.
\newblock Groups of piecewise linear homeomorphisms.
\newblock {\em Trans. Amer. Math. Soc.}, 332(2):477--514, 1992.

\bibitem[SW03]{MR1995624}
S.~Sidki and J.~Wilson.
\newblock Free subgroups of branch groups.
\newblock {\em Arch. Math. (Basel)}, 80(5):458--463, 2003.

\bibitem[SWZ19]{MR3910073}
R.~Skipper, S.~Witzel, and M.~C.~B. Zaremsky.
\newblock Simple groups separated by finiteness properties.
\newblock {\em Invent. Math.}, 215(2):713--740, 2019.

\bibitem[Zar22]{twistedV2}
M.~C.~B. Zaremsky.
\newblock A taste of twisted {B}rin--{T}hompson groups.
\newblock {\em arXiv preprint arXiv:2201.00711}, 2022.

\end{thebibliography}

\normalsize

\vspace{0.5cm}

\textsc{Department of Mathematics, University at Buffalo (SUNY)}

\textit{E-mail address:} \texttt{sahanaha@buffalo.edu} \\

\textsc{Department of Mathematics, ETH Z\"urich, Switzerland}

\textit{E-mail address:} \texttt{francesco.fournier@math.ethz.ch} \\

\textsc{Institut Montpellierain Alexander Grothendieck, Montpellier, France}

\textit{E-mail address:} \texttt{anthony.genevois@umontpellier.fr} \\

\textsc{Maxwell Institute and Department of Mathematics, Heriot-Watt University, Edinburgh, UK}

\textit{E-mail address:} \texttt{a.sisto@hw.ac.uk}

\end{document}